\documentclass[12pt]{amsart}
\usepackage{amsmath, amsthm, latexsym,amssymb}
\usepackage{xcolor}

\usepackage{hyperref}

\numberwithin{equation}{section}

\theoremstyle{plain}
\newtheorem{theor}{Theorem}[section]

\newtheorem{lemma}[theor]{Lemma}

\theoremstyle{remark}
\newtheorem{rem}[theor]{Remark}

\newcommand{\EE}{\mathbb{E\thinspace}}
\newcommand{\PP}{\mathbb{P\thinspace}}

\DeclareMathOperator{\tr}{trace}
\DeclareMathOperator{\spa}{span}
\DeclareMathOperator{\Sym}{Sym}
\DeclareMathOperator{\codim}{codim}
\DeclareMathOperator{\rk}{rank}
\DeclareMathOperator{\dist}{dist}

\begin{document}

\title{When a system of real quadratic equations has a solution}

\author[Alexander Barvinok]{Alexander Barvinok}
\address{Department of Mathematics, University of Michigan, Ann Arbor, MI 48109-1043, USA }
\email{ \{barvinok, rudelson\}@umich.edu}
\author[Mark Rudelson]{Mark Rudelson}
\date{October 3, 2021} 
\thanks{  Research of AB and MR is partially supported by NSF Grants  DMS 1855428 and DMS  2054408 respectively.} 

   
 \begin{abstract}
We provide a sufficient condition for solvability of a system of real quadratic equations  $p_i(x)=y_i$, $i=1, \ldots, m$, 
where $p_i: {\mathbb R}^n \longrightarrow {\mathbb R}$ are quadratic forms.
By solving a positive semidefinite program,  one can reduce it to another system of the type $q_i(x)=\alpha_i$, $i=1, \ldots, m$, 
where $q_i: {\mathbb R}^n \longrightarrow {\mathbb R}$ are quadratic forms and $\alpha_i=\tr q_i$. 
We prove that the latter system has solution
$x \in {\mathbb R}^n$ if for some (equivalently, for any) orthonormal basis $A_1,\ldots, A_m$ in the space spanned by the matrices of the forms $q_i$, the operator norm 
of $A_1^2 + \ldots + A_m^2$ does not exceed $\eta/m$ for some absolute constant $\eta > 0$. 
 The condition can be checked in polynomial time and is satisfied, for example, for random $q_i$ provided $m \leq \gamma \sqrt{n}$ for an absolute constant $\gamma >0$. We prove a similar sufficient condition for a system of homogeneous quadratic equations to have a non-trivial solution.
 While the condition we obtain is of an algebraic nature, the proof relies on analytic tools including Fourier analysis and measure concentration.
 \end{abstract}

 \maketitle

\section{Introduction and main results}\label{sec1}

\subsection{Systems of real quadratic equations}\label{subsec1.1} Let $q_1, \ldots, q_m: {\mathbb R}^n \longrightarrow {\mathbb R}$ be quadratic forms, 
$$q_i(x)= \langle Q_i x, x \rangle \quad \text{for} \quad i=1, \ldots, m,$$
where $Q_i$ are $n \times n$ symmetric matrices and 
$$\langle x, y \rangle = \sum_{i=1}^n \xi_i \eta_i \quad \text{for} \quad x=\left(\xi_1, \ldots, \xi_n\right) \quad \text{and} \quad y=\left(\eta_1, \ldots, \eta_n\right)$$
is the standard scalar product in ${\mathbb R}^n$. 

Let $\alpha_1, \ldots, \alpha_m$ be real numbers. We want to find out when the system of equations 
\begin{equation}\label{eq1.1.1}
q_i(x) = \alpha_i \quad \text{for} \quad i=1, \ldots, m 
\end{equation}
has a solution $x \in {\mathbb R}^n$. Such systems of equations appear in various contexts, see, for example, \cite{Bi16}, \cite{L+14}, \cite{Pa13}. If the number $m$ of equations is fixed in advance, one can decide in polynomial time whether the system has a solution \cite{Ba93}, \cite{GP05}, \cite{Ba08}. The same is true if the number $n$ of variables is fixed in advance, in which case a polynomial time algorithm to test feasibility exists even if $q_i$ are polynomials of an arbitrary degree, see, for example, \cite{B+06}. 

If $m$ and $n$ are both allowed to grow, the problem becomes computationally hard. Unless the computational complexity hierarchy collapses, there is no polynomial time algorithm to test the feasibility of \eqref{eq1.1.1}. Furthermore, it is not known whether the feasibility problem belongs to the complexity class {\bf NP}. In other words, it is not known
whether one can present a polynomial size certificate for the system \eqref{eq1.1.1} to have a solution when it is indeed feasible (note that using repeated squaring of the type $x_{n+1}=x_n^2$, one can construct examples of feasible systems for which no solution has a polynomial size description).

In fact, testing the feasibility of an arbitrary system of real polynomial equations can be easily reduced to testing the feasibility of a system (\ref{eq1.1.1}). First, we gradually reduce the degree of polynomials by introducing new variables and equations of the type $\xi_{ij}-\xi_i \xi_j=0$, and hence reduce a given polynomial system to a system 
$$q_i(x)=0 \quad \text{for} \quad i=1, \ldots, m,$$
where $q_i$ are quadratic, not necessarily homogeneous, polynomials. Then we introduce  another variable $\tau$ and replace the above system by a system of homogeneous 
quadratic equations 
$$\tau^2 q_i\left(\tau^{-1} x \right)=0 \quad \text{for} \quad i=1, \ldots, m$$ with one more quadratic constraint $\tau^2=1$.

We are also interested in systems of homogeneous equations
\begin{equation}\label{eq1.1.2}
q_i(x)=0 \quad \text{for} \quad i=1, \ldots, m, 
\end{equation}
in which case we want to find out whether the system has a non-trivial solution $x \ne 0$.  The problem is also computationally hard. We briefly sketch how an efficient algorithm for testing the existence of a non-trivial solution in \eqref{eq1.1.2} would produce an efficient algorithm for testing the feasibility of \eqref{eq1.1.1}. Given a system \eqref{eq1.1.1}, by introducing a new variable $\tau$, as above we replace \eqref{eq1.1.1} by a system of homogeneous quadratic equations, where we want to enforce $\tau \ne 0$. This is done by introducing yet another variable $\sigma$ and the equation 
$$R^2 \tau^2 - \left(\xi_1^2 + \ldots + \xi_n^2\right)=\sigma^2$$
binding all variables together, so that if $\tau =0$ then all other variables are necessarily $0$. Here $R$ is meant to be a very large constant and in fact, it can be treated as infinitely large, with computations in the ordered field of rational functions in $R$, the trick first introduced in \cite{GV88}.

In this paper, we present a computationally simple sufficient criteria for \eqref{eq1.1.1}, respectively \eqref{eq1.1.2}, to have a solution, respectively a non-trivial solution. We 
start with by now a standard procedure of semidefinite relaxation.

\subsection{Positive semidefinite relaxation}\label{subsec1.2}
For an $n \times n$ real symmetric matrix $X$, we write $X \succeq 0$ to say that $X$ is positive semidefinite. 

Given (\ref{eq1.1.1}), we consider the following system of linear equations 
\begin{equation}\label{eq1.2.1}
\tr (Q_i X) = \alpha_i \quad \text{for} \quad i=1, \ldots, m \qquad \text{where} \qquad
X \succeq 0 
\end{equation}
in $n \times n$ positive semidefinite matrices $X$. Unlike (\ref{eq1.1.1}), the system (\ref{eq1.2.1}) is convex and efficient algorithms are available to test its feasibility, see \cite{Pa13} for a survey.
Clearly, if $x=\left(\xi_1, \ldots, \xi_n\right)$ is a solution to (\ref{eq1.1.1}) then the matrix $X=\left(x_{ij}\right)$ defined by 
$x_{ij}=\xi_i \xi_j$ is a positive semidefinite solution to (\ref{eq1.2.1}). If $m \leq 2$, then the converse is true: if the system (\ref{eq1.2.1}) has a solution then so does (\ref{eq1.1.1}), see, for example,
Section II.13 of \cite{Ba02a}. For $m \geq 3$ the system (\ref{eq1.2.1}) may have solutions while (\ref{eq1.1.1}) may be infeasible. For example, the system of quadratic equations 
$$\xi_1^2 =1, \quad \xi_2^2=1 \quad \text{and} \quad \xi_1 \xi_2 = 0$$ 
does not have a solution, whereas the $2 \times 2$ identity matrix $I$ is the solution to its positive semidefinite relaxation. One corollary of our results is that such examples are, in some sense, ``atypical".

Our goal is to find a computationally simple criterion when a solution to (\ref{eq1.2.1}) implies the existence of a solution to (\ref{eq1.1.1}).

Let $X$ be a solution to (\ref{eq1.2.1}). Since $X \succeq 0$, we can write $X=T T^{\ast}$ for an $n \times n$ matrix $T$. Then 
$$ \tr(Q_i X)= \tr (Q_i T T^{\ast}) = \tr (T^{\ast} Q_i T).$$
Let us define matrices
\begin{equation}\label{eq1.2.2}
\widehat{Q_i} = T^{\ast} Q_i T  \quad \text{for} \quad i=1, \ldots, m
\end{equation}
and the corresponding quadratic forms $\widehat{q_i}: {\mathbb R}^n \longrightarrow {\mathbb R}$, 
\begin{equation}\label{eq1.2.3}
\widehat{q_i}(x)=\langle  \widehat{Q_i} x, x \rangle = q_i(Tx) \quad \text{for} \quad i=1, \ldots, m. 
\end{equation}
If $x \in {\mathbb R}^n$ is a solution to the system 
\begin{equation}\label{eq1.2.4}
\widehat{q_i}(x)= \alpha_i \quad \text{for} \quad i=1, \ldots, m
\end{equation}
then $y=Tx$ is a solution to (\ref{eq1.1.1}). We note that 
\begin{equation}\label{eq1.2.5}
\alpha_i= \tr \widehat{Q_i} \quad \text{for} \quad i=1, \ldots, m. 
\end{equation}
 It may happen that the system (\ref{eq1.1.1}) has a solution while (\ref{eq1.2.4}) does not, but if 
$X$ and hence $T$ are invertible, the systems (\ref{eq1.1.1}) and (\ref{eq1.2.4}) are equivalent. Furthermore, if there are no invertible $X \succeq 0$ satisfying (\ref{eq1.2.1}), then the affine subspace defined by the equations $\tr (Q_i X)=\alpha_i$ intersects the cone of positive semidefinite matrices at a proper face, and the system (\ref{eq1.1.1}) can be effectively reduced to a system of quadratic equations in fewer variables, cf., for example, Section II.12 of \cite{Ba02a}. Summarizing, a solution $X$ to (\ref{eq1.2.1}) allows us to replace (\ref{eq1.1.1}) by a similar system, where the right hand sides $\alpha_i$ are the traces of the quadratic forms in the left hand side.

Ultimately, we are interested in finding out when the system (\ref{eq1.2.4}) of quadratic equations with additional conditions (\ref{eq1.2.5}) has a solution $x \in {\mathbb R}^n$.

\subsection{Reduction to an orthonormal basis and the main result}\label{subsec1.3}
Before we state our main result, some remarks are in order. As agreed, we consider the system (\ref{eq1.1.1}) where $\alpha_i=\tr q_i$. Without loss of generality, we assume that 
the quadratic forms $q_i$ and hence their matrices $Q_i$ are linearly independent. For an invertible $m \times m$ matrix $M=\left(\mu_{ij}\right)$, let us define new forms 
$$\tilde{q}_i=\sum_{j=1}^m \mu_{ij} q_j \quad \text{for} \quad i=1, \ldots, m$$
and new right hand sides 
$$\tilde{\alpha_i} = \sum_{j=1}^m \mu_{ij} \alpha_j \quad \text{for} \quad i=1, \ldots, m.$$
Then the system (\ref{eq1.1.1}) has a solution if and only if the system 
$$\tilde{q}_i(x)=\tilde{\alpha}_i \quad \text{for} \quad i=1, \ldots, m$$
has a solution. Hence, ideally, a criterion for the system (\ref{eq1.1.1}) to have a solution should depend not on the forms $q_1, \ldots,  q_m$ per se (or their matrices $Q_1, \ldots, Q_m$) but on the subspace $\spa \left( q_1, \ldots, q_m\right)$ in the space of quadratic forms (equivalently, on the subspace $\spa \left(Q_1, \dots, Q_m \right)$ in the space of $n \times n$ real symmetric matrices).

We consider the standard inner product in space of $n \times n$ real matrices:
$$\langle X, Y \rangle = \tr X^{\ast} Y.$$
In particular, for symmetric matrices $X=\left(\xi_{ij}\right)$ and $Y=\left(\eta_{ij}\right)$ we have 
$$\langle X, Y \rangle = \tr XY =\sum_{1 \leq i, j \leq n} \xi_{ij} \eta_{ij}$$
and the space of $n \times n$ symmetric matrices becomes a Euclidean space. 

We will be using the following observation. Let ${\mathcal L}$ be a subspace in the space 
of $n \times n$ symmetric matrices and let $A_1, \ldots, A_m$ be an orthonormal basis of ${\mathcal L}$, so that 
$$\langle A_i, A_j \rangle = \tr A_i A_j= \begin{cases} 1 &\text{if\ } i =j , \\ 0 &\text{if\ } i \ne j. \end{cases}$$
Then the matrix $A_1^2 + \ldots + A_m^2$ does not depend on a choice of an orthonormal basis and hence is an invariant of the subspace ${\mathcal L}$. Indeed, if $B_1, \ldots, B_m$ is another orthonormal basis of ${\mathcal L}$, then 
$$B_i=\sum_{j=1}^m \mu_{ij} A_j \quad \text{for} \quad i=1, \ldots, m$$
and some orthogonal matrix $M=\left(\mu_{ij}\right)$ and hence 
$$\sum_{i=1}^m B_i^2 = \sum_{i=1}^m \left( \sum_{1 \leq j_1, j_2 \leq m} \mu_{ij_1} \mu_{i j_2} A_{j_1} A_{j_2}\right)= \sum_{1 \leq j_1, j_2 \leq m} 
\left( \sum_{i=1}^m \mu_{ij_1} \mu_{ij_2} \right) A_{j_1} A_{j_2}= \sum_{j=1}^m A_j^2.$$

For an $n \times n$ real symmetric matrix $Q$, we denote by $\| Q\|_{\mathrm{op}}$ the operator norm of $Q$, that is, the largest absolute value of an eigenvalue of $Q$.

We prove the following main result.
\begin{theor}\label{th1.3} There is an absolute constant $\eta >0$ 
such that the following holds. Let $Q_1, \ldots, Q_m$, $m \geq 3$,  be linearly independent $n \times n$ symmetric matrices and let 
$q_i: {\mathbb R}^n \longrightarrow {\mathbb R}$ for $i=1, \ldots, m$ be the corresponding quadratic forms,
$$q_i(x)=\langle Q_i x, x \rangle \quad \text{for} \quad i=1, \ldots, m.$$
Suppose that 
$$\left\| \sum_{i=1}^m A_i^2 \right\|_{\mathrm{op}} \ \leq \ \frac{\eta}{m}$$
for some (equivalently, for any) orthonormal basis $A_1, \ldots, A_m$ of the subspace 
\newline $\spa\left(Q_1, \ldots, Q_m\right)$.
Then the system of quadratic equations 
$$q_i(x)=\tr Q_i \quad \text{for} \quad i=1, \ldots, m$$
has a solution $x \in {\mathbb R}^n$.
\end{theor}

We prove a similar result for systems of homogeneous quadratic equations, where we are interested in finding a non-trivial solution.

\begin{theor}\label{th1.5} There is an absolute constant $\eta >0$  such that the following holds. 
Let $Q_1, \ldots, Q_m$, $m \geq 3$, be $n \times n$ real symmetric matrices such that 
$$\tr Q_i=0 \quad \text{for} \quad i=1, \ldots, m,$$
and 
 let $q_i: {\mathbb R}^n \longrightarrow {\mathbb R}$, 
$$q_i(x)=\langle Q_i x , x \rangle \quad \text{for} \quad i=1, \ldots, m,$$
be the corresponding quadratic forms.
Suppose that 
$$\left\| \sum_{i=1}^m A_i^2 \right\|_{\mathrm{op}} \ \leq \ \frac{\eta}{m}$$
for some (equivalently, for any) orthonormal basis $A_1, \ldots, A_m$ of the subspace \newline $\spa\left(Q_1, \ldots, Q_m\right)$.
Then the system (\ref{eq1.1.2}) of equations 
has a solution $x \ne 0$. 
\end{theor}

\begin{rem} \label{rem: eta}
Our proofs of Theorems \ref{th1.3} and \ref{th1.5} work for $\eta=10^{-6}$, however, we made no effort to optimize this constant.
 \end{rem}
 
Note that the operator norm of the matrix $\sum_{i=1}^m A_i^2$ is its largest eigenvalue. Thus, the  criterion appearing in Theorems \ref{th1.3} and \ref{th1.5} is algebraic like the problem itself. Despite that, the proofs of these theorems rely on analytic tools:  introduction of the Gaussian measure, the Fourier transform asymptotic, and the measure concentration. We discuss this in more detail in Section \ref{sec3}.

\subsection{Discussion}\label{subsec1.4} 

\subsubsection{Computational complexity}\label{subsubsec1.4.1} Given matrices $Q_1, \ldots, Q_m$, one can compute an orthonormal basis $A_1, \ldots, A_m$ of 
$\spa\left(Q_1, \ldots, Q_m\right)$, using, for example, the Gram-Schmidt orthogonalization process. Then one can check the inequality for the operator norm 
of $A_1^2 + \ldots + A_m^2$. These are standard linear algebra problems that can be solved in polynomial time. However, we don't know how to find a solution $x$ in polynomial time or whether a solution $x$ with a {\it polynomial size description} even exists when the conditions of Theorems \ref{th1.3} and \ref{th1.5} are satisfied.

\subsubsection{The case of random matrices}\label{subsubsec1.4.2} 
 Let $Q_1, \ldots, Q_m$ be independent symmetric random matrices with entries above the diagonal being independent normal random variables of expectation $0$ and variance $1$ and the diagonal entries being normal of  expectation $0$ and variance $2$. In other words, up to the scaling factor of $\sqrt{n}$, the matrices $Q_1, \ldots, Q_m$ are sampled independently from the Gaussian Orthogonal Ensemble (GOE).

We assume that $m \leq n$.
 As $n$ grows, with high probability 
we have (we ignore low-order terms)
$$\|Q_i\|_{\mathrm{op}}  \approx 2\sqrt{n} \quad \text{and} \quad  \langle Q_i, Q_i \rangle \approx n^2 \quad \text{for} \quad i=1, \ldots, m,$$
see, for example, Section 2.3 of \cite{Ta12}.

 Let $A_1, \ldots, A_m$ be the orthonormal basis of $\spa\left(Q_1, \ldots, Q_m\right)$ obtained by the Gram - Schmidt orthogonalization from 
$Q_1, \ldots, Q_m$. Then, up to a normalizing factor, each $A_i$ is also sampled from GOE, so we have 
$$\|A_i\|_{\mathrm{op}} \approx \frac{2}{\sqrt{n}} \quad \text{for} \quad i=1, \ldots, m,$$
with high probability. Hence
$$\left\| \sum_{i=1}^m A_i^2 \right\|_{\mathrm{op}} \ \leq \ \sum_{i=1}^m \|A_i\|^2_{\mathrm{op}} \ \approx \ \frac{4m}{n}.$$
Hence if $m \leq \sqrt{\eta n}/2$, with high probability the conditions of Theorems \ref{th1.3} and \ref{th1.5} are satisfied. Similar behavior can be observed for other models of random symmetric matrices with independent entries sampled from a distribution with expectation 0, variance 1 and sub-Gaussian tail. Informally, for the conditions of Theorems \ref{th1.3} and \ref{th1.5} to hold, we want $n$ to be substantially larger than $m$ and the subspace $\spa\left(Q_1, \ldots, Q_m\right)$ to be sufficiently generic.

\subsubsection{The metric geometry of the cone of positive semidefinite matrices}\label{subsubsec1.4.3} As before, we consider the space $\Sym_n$ of $n \times n$ symmetric matrices as a Euclidean space. Let ${\mathcal S}_+ \subset \Sym_n$ be the convex cone of positive semidefinite matrices. From Section \ref{subsubsec1.4.2},
we deduce the following metric property of ${\mathcal S}_+$: There is an absolute constant $\gamma > 0$ such that if ${\mathcal A} \subset \Sym_n$ is a random affine subspace with $\codim A \leq \gamma \sqrt{n}$ containing the identity matrix $I_n$, then ${\mathcal A}$ contains a positive semidefinite matrix of rank 1 with probability approaching $1$ as $n$ grows. 

We don't know if the estimates of Theorem \ref{th1.3} and Sections \ref{subsubsec1.4.2} and \ref{subsubsec1.4.3} are optimal, or, for example, whether we can make $m$ in
Section \ref{subsubsec1.4.2} and $\codim {\mathcal A}$ in Section \ref{subsubsec1.4.3} proportional to $n$ instead of $\sqrt{n}$. There is a vast literature on the {\it average} characteristics of the set of solutions for systems of real polynomial equations, see, for example, \cite{Bu07} and reference therein, but much less appears to be known regarding solvability of such systems with high probability. 

\subsubsection{Solving positive semidefinite relaxation}\label{subsubsec1.4.4} Suppose we want to apply Theorem \ref{th1.3} to test the solvability of the original system (\ref{eq1.1.1}), where we do not necessarily have $\alpha_i = \tr q_i$. We begin by looking for a solution $X$ to the positive semidefinite program (\ref{eq1.2.1}).
If there is no solution $X$, we conclude that the system (\ref{eq1.1.1}) has no solutions. If there is a solution $X \succeq 0$ with $\rk X \leq 1$, we conclude that the system 
(\ref{eq1.1.1}) has a solution. The difficulty arises when we find a solution $X \succeq 0$ but with $\rk X > 1$. It is known that if there is a solution $X \succeq 0$, then there is a solution $X \succeq 0$ with an additional constraint 
$$\rk X \ \leq \ \left\lfloor \frac{\sqrt{8m+1} -1}{2}\right\rfloor.$$
Any extreme point of the set of solutions to (\ref{eq1.2.1}) satisfies this condition, see, for example, Section II.13 of \cite{Ba02a}. Curiously, if we are to use Theorem \ref{th1.3}
to ascertain the existence of a solution, it makes sense to try to find an $X \succeq 0$ not on the boundary, but as close as possible to the ``middle" of the set of solutions of 
(\ref{eq1.2.1}) because we want the transformed matrices $\widehat{Q}_i$ given by (\ref{eq1.2.2}) to be as generic as possible. For example, one can look for $X$ with the maximum von Neumann entropy 
$$\sum_{j=1}^n \lambda_j \ln \frac{1}{\lambda_j},$$
where $\lambda_1, \ldots, \lambda_n$ are the eigenvalues of $X$, see, for example, \cite{We78}.  Finding such an $X$ is a convex optimization problem and hence can be solved efficiently. Informally, if the number $m$ of equations is rather small compared to the number $n$ of variables, if the 
matrices $Q_1, \ldots, Q_m$ of equations in (\ref{eq1.1.1}) are sufficiently generic, and if the solutions $X$ to the positive semidefinite relaxation (\ref{eq1.2.1}) can be found deep enough the cone of ${\mathcal{S}_+}$ positive semidefinite matrices, then the system (\ref{eq1.1.1}) will have a solution.

We note that in the homogeneous case one should also
 be careful about working with the positive semidefinite relaxation. Namely, if $X \succeq 0$ is a solution to the system of equations
\begin{equation}\label{eq1.5.1}
\tr (Q_i X) = 0 \quad \text{for} \quad i=1, \ldots, m,
\end{equation}
we factor $X=TT^{\ast}$, define $\widehat{Q_i}$ by (\ref{eq1.2.2}) and define $\widehat{q_i}$ by (\ref{eq1.2.3}), then to deduce the existence of a non-trivial solution to the system (\ref{eq1.1.2}) from the existence of a non-trivial solution to the system 
$$\widehat{q_i}(x)=0 \quad \text{for} \quad i=1, \ldots, m,$$
we must require $T$ and hence $X$ to be invertible. If there are no invertible $X \succeq 0$ satisfying (\ref{eq1.5.1}), we reduce (\ref{eq1.1.2}) to a system of homogeneous quadratic equations in fewer variables, see Section \ref{subsec1.2}.

In the rest of the paper, we prove Theorems \ref{th1.3} and \ref{th1.5}. Although the statements are real algebraic, our proofs use analytic methods, in particular, the Fourier transform. 

\section{ Outline of the proof}\label{sec3}

In what follows, we denote the imaginary unit by $\sqrt{-1}$, so as to use $i$ for indices. 

Let $Q_1, \ldots, Q_m$ be $n \times n$ real symmetric matrices and let $I$ be the $n \times n$ identity matrix. For real $\tau_1, \ldots, \tau_m$, we consider 
the matrix 
$$Q(t)=I - \sqrt{-1} \sum_{i=1}^m \tau_i Q_i \quad \text{for} \quad t=\left(\tau_1, \ldots, \tau_m \right).$$
Since the eigenvalues $\lambda_1(t), \ldots, \lambda_n(t)$ of the linear combination $\sum_{i=1}^m \tau_i Q_i$ are real, we have
$$\det Q(t) = \prod_{i=1}^n \left(1 - \sqrt{-1} \lambda_i(t)\right) \ne 0 \quad \text{for all} \quad t \in {\mathbb R}^m.$$
Therefore, we can pick a branch of 
$$\det^{\quad  -\frac{1}{2}} Q(t),$$
which we select in such a way so that at $t=0$ we get $1$.

It is also more convenient to rescale and define quadratic forms by 
$$q(x)=\frac{1}{2} \langle Qx, x \rangle.$$
Our proof of Theorems \ref{th1.3} hinges on 
the analysis of the  Fourier transform of the function $F(t):=\det^{-\frac{1}{2}} Q(t), \ t \in \mathbb{R}^m$. Namely, we prove
the following result.

\begin{theor}\label{th2.1} Let $Q_1, \ldots, Q_m$ be $n \times n$ real symmetric matrices, let
$$q_i(x)=\frac{1}{2} \langle Q_i x, x \rangle \quad \text{for} \quad i=1, \ldots, m,$$
be the corresponding quadratic forms and let  $\alpha_1, \ldots, \alpha_m$ be real numbers. Suppose that 
\begin{equation}\label{eq2.1.1}
\int_{{\mathbb R}^m} \left| \det^{\quad - \frac{1}{2}} \left(I -\sqrt{-1} \sum_{i=1}^m \tau_i Q_i \right)\right| \ dt \ < \ +\infty 
\end{equation}
and that 
\begin{equation}\label{eq2.1.2}
\int_{{\mathbb R}^m} \det^{\quad -\frac{1}{2}} \left(I -\sqrt{-1} \sum_{i=1}^m \tau_i Q_i \right) \exp\left\{ -\sqrt{-1} \sum_{i=1}^m \alpha_i \tau_i \right\} \ dt \ \ne  0.
\end{equation}
Then the system (\ref{eq1.1.1}) of equations has a solution $x \in {\mathbb R}^n$.
\end{theor}

We prove a similar result for homogeneous systems.
\begin{theor}\label{th2.4} Let $Q_1, \ldots, Q_m$ and $q_1, \ldots, q_m$ be as in Theorem \ref{th2.1} and assume, additionally, that $m < n$.
Suppose that 
$$ \int_{{\mathbb R}^m} \det^{\quad -\frac{1}{2}} \left(I -\sqrt{-1} \sum_{i=1}^m \tau_i Q_i \right) \ dt \ \ne  0,$$
where the integral converges absolutely. Then the system (\ref{eq1.1.2})  of equations 
has a solution $x \ne 0$.
\end{theor}

We prove Theorems \ref{th2.1} and \ref{th2.4} in Section \ref{sec2}. 
 Theorems \ref{th1.3} and \ref{th1.5} are deduced from Theorems \ref{th2.1} and \ref{th2.4} respectively. Since the proofs are very similar, below we discuss the plan of the proof of Theorem \ref{th1.3} only.

First, we note that we can replace matrices $Q_1, \ldots, Q_m$ by an orthonormal set of matrices $A_1, \ldots, A_m$ and quadratic forms $q_i$ by quadratic forms 
$$a_i(x)=\frac{1}{2} \langle A_i x, x \rangle \quad \text{for} \quad i=1, \ldots, m.$$
We let 
$$\alpha_i =\frac{1}{2} \tr A_i$$
and consider an equivalent system 
$$a_i(x)=\alpha_i \quad \text{for} \quad i=1, \ldots, m$$
of quadratic equations, see Section \ref{subsec1.3}.

Using Theorem \ref{th2.1}, we conclude that it suffices to prove that 
\begin{equation}\label{eq3.1}
\int_{{\mathbb R}^m} \det^{\quad -\frac{1}{2}} \left(I - \sqrt{-1} \sum_{i=1}^m \tau_i A_i\right) \exp\left\{ -\sqrt{-1} \sum_{i=1}^m \alpha_i \tau_i \right\} \ dt \ne 0,
\end{equation}
where the integral converges absolutely. Up to a scaling normalization factor, we rewrite the integral in polar coordinates as follows.

Let ${\mathbb S}^{m-1} \subset {\mathbb R}^m$ be the unit sphere endowed with the Haar probability measure. For $w \in {\mathbb S}^{m-1}$, $w=\left(\omega_1, \ldots, \omega_m\right)$, we define the matrix 
$$A(w)=\sum_{i=1}^m \omega_i A_i.$$ Up to a non-zero scaling factor, in polar coordinates the integral (\ref{eq3.1}) can be written as 
\begin{equation}\label{eq3.2}
\int_{{\mathbb S}^{m-1}} \left(\int_0^{+\infty} \tau^{m-1} \det^{\quad -\frac{1}{2}}\left(I - \sqrt{-1} \tau A(w)\right) \exp\left\{ -\frac{\sqrt{-1} \tau}{2} \tr A(w) \right\} \ d \tau \right) dw.
\end{equation}
The rest of the proof relies on an analysis of this integral.
As a first step, we show that the contribution of the tail  of the inside integral in (\ref{eq3.2}) is negligible. Namely, we prove in Lemma \ref{le5.1} that for any $w \in {\mathbb S}^{m-1}$, we have 
\begin{equation}\label{eq3.3}
\int_{5 \sqrt{m}}^{+\infty} \tau^{m-1} \left| \det^{\quad -\frac{1}{2}}\left( I - \sqrt{-1} \tau A(w)\right)\right| \ d\tau \ \leq \ \frac{1}{20 m} m^{m/2} e^{-3m}.
\end{equation}
In particular, this proves that the integral (\ref{eq3.2}) converges absolutely and that the integrals (\ref{eq3.2}) and (\ref{eq3.1}) are equal, up to a scaling factor that is the surface 
area of the unit sphere ${\mathbb S}^{m-1} \subset {\mathbb R}^m$.

This allows us to consider  the integration over the interval $[0,5 \sqrt{m}]$ in the inner integral in \eqref{eq3.2}.
To analyze this integral, denote by $\lambda_1(w), \ldots, \lambda_n(w)$ the eigenvalues of $A(w)$. 
A simple calculation  yields
\begin{align*}
  &\det^{\quad -\frac{1}{2}}\left(I - \sqrt{-1} \tau A(w)\right) \exp\left\{ -\frac{\sqrt{-1} \tau}{2} \tr A(w) \right\} \\
  &=  \exp\left\{ \frac{1}{2} \sum_{k=2}^{\infty} \frac{(\tau \sqrt{-1})^k}{k} \sum_{j=1}^n \lambda_j^k(w) \right\},
\end{align*}
see the derivation in \eqref{eq: det}. 
Note that the summation starts from $k=2$. This is achieved due to the first step in the argument allowing us to set $\alpha_i =\frac{1}{2} \tr A_i$.
Moreover, $\sum_{j=1}^n \lambda_j^2(w)=1$ for all $w \in {\mathbb S}^{m-1}$ due to orthonormality of the matrices $A_1, \ldots, A_m$.

Next, we divide the points $w \in {\mathbb S}^{m-1}$ into tame and wild. For a tame point, we show that the term corresponding to $k=2$ in the expression above is dominating which would mean that the expression above is close to $\exp\left\{ -\frac{\tau^2}{4}\right\} $.
To prove it, we need to control $\sum_{j=1}^n \lambda_j^k(w)$ for all $k \ge 3$.
However, as we show below, a control for $k=3,4$ turns out to be sufficient. 
More precisely, we classify a point $w \in {\mathbb S}^{m-1}$ as tame if
\[
  \left| \sum_{j=1}^n \lambda_j^3(w)\right| \ \leq \ \frac{1}{25m^{3/2}}  \quad \text{and} \quad \sum_{j=1}^n \lambda_j^4(w) \ \leq \ \frac{1}{625 m^2}.
\]
The second inequality here is a bound on the 4-Schatten norm of $A(w)$: $\|A(w)\|_{\mathrm{S_4}} \le 1/(625 m^2)$. 
In contrast to it, the first inequality bounds the third moment of the eigenvalues, and not the 3-Schatten norm, as we have to exploit the cancellation of positive and negative eigenvalues.

In Lemma \ref{le5.2}, we prove that if $w \in {\mathbb S}^{m-1}$ is tame, then
\begin{equation}\label{eq3.4}
\begin{split} 
&\Re\thinspace \int_0^{5\sqrt{m}} \tau^{m-1}\det^{\quad -\frac{1}{2}}\left( I - \sqrt{-1} \tau A(w)\right) \exp\left\{ -\frac{\sqrt{-1} \tau}{2} \tr A(w)\right\} \ 
d \tau \\ \geq \ &\frac{1}{2} \int_0^{5 \sqrt{m}} \tau^{m-1} \exp\left\{ -\frac{\tau^2}{4}\right\} \ d \tau \approx  2^{m-2} \Gamma\left(\frac{m}{2}\right).
\end{split}
\end{equation}

We note that the value of (\ref{eq3.4}) is much larger than the tail estimate (\ref{eq3.3}). Moreover, in Lemmas \ref{le4.2} and \ref{le4.3} we bound the expectations
\begin{equation}\label{eq3.5}
\EE \left(\sum_{j=1}^n \lambda_j^3(w)\right)^2 \ \leq \ \frac{120 \eta}{m(m+2)(m+4)} \quad \text{and} \quad \EE \sum_{j=1}^n \lambda_j^4(w) \ \leq \ \frac{3 \eta}{(m+2) m}.
\end{equation}
This is the point where the quantity $\left\| \sum_{i=1}^m A_i^2 \right\|_{\mathrm{op}}$ reveals itself. 
It turns out that both expectations above can be controlled in terms of this operator norm alone.

It follows then by the Markov inequality that a random $w \in {\mathbb S}^{m-1}$ is tame with probability at least $7/8$, and hence tame points $w \in {\mathbb S}^{m-1}$ contribute significantly to the integral (\ref{eq3.2}).

It remains to show that the contribution of wild points $w \in {\mathbb S}^{m-1}$ cannot offset the contribution of tame points. 

This relies on a concentration inequality for the 4-Schatten norm of matrices $A(w)$ on the unit sphere ${\mathbb S}^{m-1}$, which we derive in Lemma \ref{le4.5}. This inequality is leveraged against the deterioration of the bounds on the eigenvalues of $A(w)$ which occurs for the wild points.  
To this end, we partition the set of wild points into a number of subsets according to the size of $\|A(w)\|_{\mathrm{S_4}}$, and apply the concentration inequality to prove that the contribution of the points in each layer to the integral \eqref{eq3.2} is negligible.

This argument is carried out in Section \ref{sec6}.

 In Sections \ref{sec4} and \ref{sec5}, we do some preliminary work: we prove bounds (\ref{eq3.5}) as well as some other useful bounds on the eigenvalues of $A(w)$ in Section \ref{sec4}.
 In Section \ref{sec5} we derive \eqref{eq3.3} and show that a similar integral over the interval $[0,5 \sqrt{m}]$ can be controlled by $\|A(w)\|_{\mathrm{op}}$, which is in turn bounded in terms of $\|A(w)\|_{\mathrm{S_4}}$. 
 
 \section{Proofs of Theorems \ref{th2.1}  and \ref{th2.4}}\label{sec2}

\subsection{Enter Gaussian measure}\label{subsec2.1} We consider the standard Gaussian measure in ${\mathbb R}^n$ with density 
$$\frac{1}{(2 \pi)^{n/2}} e^{-\|x\|^2/2} \quad \text{where} \quad \|x\|=\sqrt{\xi_1^2 + \ldots + \xi_n^2} \quad \text{for} \quad x=\left(\xi_1, \ldots, \xi_n \right).$$
Considering a quadratic form $q(x) =\langle Qx, x \rangle$ as a random variable, we observe that 
$$\EE q=\tr Q,$$
so that the equation $q(x)=\tr Q$ ``holds on average".

The proof of Theorems \ref{th2.1} and \ref{th2.4} is based on a Fourier transform formula.
\begin{lemma}\label{le2.2} Let $Q_1, \ldots, Q_m$ be $n \times n$ real symmetric matrices and let 
$$q_i(x)=\frac{1}{2} \langle Q_i x, x \rangle \quad \text{for} \quad i=1, \ldots, m,$$
be the corresponding quadratic forms. Then for any real $\alpha_1, \ldots, \alpha_m$ and any real $\sigma > 0$, we have 
\begin{equation*}
\begin{split} & \frac{1}{(2 \pi)^{n/2}} \int_{{\mathbb R}^n} \exp\left\{ -\frac{\sigma^2}{2} \sum_{i=1}^m \left(q_i(x)-\alpha_i\right)^2 \right\} e^{-\|x\|^2/2} \ dx\\= 
&\frac{1}{\sigma^m (2\pi)^{m/2}} \int_{{\mathbb R}^m} \det^{\quad -\frac{1}{2}} \left(I -\sqrt{-1} \sum_{i=1}^m \tau_i Q_i \right) \exp\left\{ -\sqrt{-1} \sum_{i=1}^m \alpha_i \tau_i \right\}
e^{-\|t\|^2/2 \sigma^2} \ dt. 
\end{split}
\end{equation*}
\end{lemma}
\begin{proof} As is well-known, for a positive definite matrix $Q$ and the corresponding form 
$$q(x)=\frac{1}{2} \langle Qx, x \rangle$$
we have 
$$\frac{1}{(2 \pi)^{n/2}} \int_{{\mathbb R}^n} e^{-q(x)} \ dx = \det^{\quad -\frac{1}{2}} Q.$$
Consequently, for $t \in{\mathbb R}^m$, $t=\left(\tau_1, \ldots, \tau_m\right)$, in a sufficiently small neighborhood of $0$, we have 
$$\frac{1}{(2 \pi)^{n/2}} \int_{{\mathbb R}^n} \exp\left\{ \sum_{i=1}^m \tau_i q_i(x) \right\} e^{-\|x\|^2/2} \ dx = \det^{\quad -\frac{1}{2}} \left(I - \sum_{i=1}^m \tau_i Q_i \right).$$
Since both sides of the formula are analytic in $\tau_1, \ldots, \tau_m \in {\mathbb C}$ for $\Re \tau_1, \ldots, \Re \tau_m$ in a small neighborhood of $0$, we conclude that the above formula holds for all such 
$\tau_1, \ldots, \tau_m$ and that, in particular,
\begin{equation*}
\begin{split} & \frac{1}{(2 \pi)^{n/2}} \int_{{\mathbb R}^n} \exp\left\{ \sqrt{-1} \sum_{i=1}^m \tau_i q_i(x) \right\} e^{-\|x\|^2/2} \ dx \\= &\det^{\quad -\frac{1}{2}} \left(I - \sqrt{-1} \sum_{i=1}^m \tau_i Q_i \right) 
\end{split}
\end{equation*}
for all real $\tau_1, \ldots, \tau_m$. 

Therefore, 
\begin{equation}\label{eq2.2.1}
\begin{split}
&\frac{1}{(2 \pi)^{n/2}} \int_{{\mathbb R}^n} \exp\left\{ \sqrt{-1} \sum_{i=1}^m \tau_i \left( q_i(x) - \alpha_i \right) \right\} e^{-\|x\|^2/2} \ dx \\= &\det^{\quad -\frac{1}{2}} \left(I - \sqrt{-1} \sum_{i=1}^m \tau_i Q_i \right) \exp\left\{-\sqrt{-1} \sum_{i=1}^m \alpha_i \tau_i \right\}
 \end{split}
 \end{equation}
for all real $\tau_1, \ldots, \tau_m$. 

Next, we use a well-known formula: for $\sigma >0$ and any real  (or complex) $\alpha$, we have 
$$\frac{1}{\sigma \sqrt{2 \pi}} \int_{-\infty}^{+\infty} \exp\left\{\sqrt{-1} \alpha \tau\right\} \exp\left\{ -\frac{\tau^2}{2 \sigma^2}\right\} \ d\tau = \exp\left\{ -\frac{\alpha^2 \sigma^2}{2}\right\}.$$
Integrating both sides of (\ref{eq2.2.1}) for $i=1, \ldots, m$ over $\tau_i \in {\mathbb R}$ with density 
$$\frac{1}{\sigma \sqrt{2 \pi}} \exp\left\{ - \frac{\tau_i^2}{2 \sigma^2} \right\},$$ 
we get the desired formula.
\end{proof}

Now we are ready to prove Theorem \ref{th2.1}.

\begin{proof}[Proof of Theorem \ref{th2.1}] By Lemma \ref{le2.2}, for all $\sigma >0$, we have 
\begin{equation}\label{eq2.3.3}
\begin{split} & \sigma^m \int_{{\mathbb R}^n} \exp\left\{ -\frac{\sigma^2 }{2} \sum_{i=1}^m \left(q_i(x)-\alpha_i\right)^2 \right\} e^{-\|x\|^2/2} \ dx=\\
&(2 \pi)^{\frac{n-m}{2}} \int_{{\mathbb R}^m} \det^{\quad -\frac{1}{2}} \left(I -\sqrt{-1} \sum_{i=1}^m \tau_i Q_i \right) \exp\left\{ -\sqrt{-1} \sum_{i=1}^m \alpha_i \tau_i \right\}
e^{-\frac{\|t\|^2}{2 \sigma^2}} \ dt.
\end{split}
 \end{equation}
As $\sigma \longrightarrow +\infty$, the right hand side of (\ref{eq2.3.3}) converges to 
$$(2 \pi)^{\frac{n-m}{2}} \int_{{\mathbb R}^m} \det^{\quad -\frac{1}{2}} \left(I -\sqrt{-1} \sum_{i=1}^m \tau_i Q_i \right) \exp\left\{ -\sqrt{-1} \sum_{i=1}^m \alpha_i \tau_i \right\} \ne 0.$$
Suppose that the system (\ref{eq1.1.1}) has no solutions $x \in {\mathbb R}^n$. We intend to obtain a contradiction by showing that the left hand side of (\ref{eq2.3.3}) converges to $0$ 
as $\sigma \longrightarrow +\infty$.

Let 
$$\gamma=(2\pi)^{\frac{n-m}{2}} \int_{{\mathbb R}^m} \left| \det^{\quad -\frac{1}{2}} \left(I -\sqrt{-1} \sum_{i=1}^m \tau_i Q_i \right)\right| \ dt \ < \ +\infty.$$
Let us choose a $\rho > 0$, to be adjusted later. Then
\begin{equation*}
\begin{split} &\sigma^m \int_{x \in {\mathbb R}^n: \ \|x\| > \rho} \exp\left\{ -\frac{\sigma^2}{2} \sum_{i=1}^m \left(q_i(x)-\alpha_i\right)^2 \right\} e^{-\|x\|^2/2} \ dx \\
\leq \ &e^{-\rho^2/4} \sigma^m \int_{x \in {\mathbb R}^n: \ \|x\| > \rho} \exp\left\{ -\frac{\sigma^2}{2} \sum_{i=1}^m \left(q_i(x)-\alpha_i\right)^2 \right\} e^{-\|x\|^2/4} \ dx \\
\leq \ &e^{-\rho^2/4} \sigma^m \int_{{\mathbb R}^n}  \exp\left\{ -\frac{\sigma^2}{2} \sum_{i=1}^m \left(q_i(x)-\alpha_i\right)^2 \right\} e^{-\|x\|^2/4} \ dx \\
=&e^{-\rho^2/4} 2^{n/2}  \sigma^m  \int_{{\mathbb R}^n} \exp\left\{ -\frac{\sigma^2}{2} \sum_{i=1}^m \left(2q_i(x)-\alpha_i\right)^2 \right\} e^{-\|x\|^2/2} \ dx  \\
=&e^{-\rho^2/4} 2^{n/2}  \sigma^m \int_{{\mathbb R}^n} \exp\left\{ -{\frac{(2\sigma)^2}{2}} \sum_{i=1}^m \left(q_i(x)-\frac{\alpha_i}{2}\right)^2 \right\} e^{-\|x\|^2/2} \ dx  \\
=&e^{-\rho^2/4} 2^{n/2} 2^{-m} (2\sigma)^m \int_{{\mathbb R}^n} \exp\left\{ -{\frac{(2\sigma)^2}{2}} \sum_{i=1}^m \left(q_i(x)-\frac{\alpha_i}{2}\right)^2 \right\} e^{-\|x\|^2/2} \ dx.
 \end{split}
 \end{equation*}
 
From Lemma \ref{le2.2}, 
\begin{equation*}
\begin{split} &(2 \sigma)^m \int_{{\mathbb R}^n} \exp\left\{ -\frac{(2\sigma)^2}{2} \sum_{i=1}^m \left(q_i(x)-\frac{\alpha_i}{2}\right)^2 \right\} e^{-\|x\|^2/2} \ dx \\   \leq \ 
& (2 \pi)^{\frac{n-m}{2}} \int_{{\mathbb R}^m} 
\left| \det^{\quad -\frac{1}{2}} \left(I - \sqrt{-1} \sum_{i=1}^m \tau_i Q_i \right) \right| \ dt =\gamma.
\end{split}
\end{equation*}

Summarizing,
\begin{equation*}
\begin{split} &\sigma^m \int_{x \in {\mathbb R}^n: \ \|x\| > \rho} \exp\left\{ -\frac{\sigma^2}{2} \sum_{i=1}^m \left(q_i(x)-\alpha_i\right)^2 \right\} e^{-\|x\|^2/2} \ dx \\ \leq \ 
&e^{-\rho^2/4} 2^{n/2} 2^{-m} \gamma. 
\end{split}
\end{equation*}
Given $\epsilon > 0$, we choose $\rho(\epsilon) > 0$ such that 
$$e^{-\rho^2(\epsilon)/4} 2^{n/2} 2^{-m} \gamma \ \leq \ \frac{\epsilon}{2},$$
so that for all $\sigma >0$ we have 
\begin{equation}\label{eq2.3.4}
\sigma^m \int_{x \in {\mathbb R}^n: \ \|x\| > \rho(\epsilon)} \exp\left\{ -\frac{\sigma^2}{2} \sum_{i=1}^m \left(q_i(x)-\alpha_i\right)^2 \right\} e^{-\|x\|^2/2} \ dx \ \leq \ \frac{\epsilon}{2}.\end{equation}

If the system (\ref{eq1.1.1}) has no solution then for some $\delta(\epsilon) >0$, we have
$$\sum_{i=1}^m \left(q_i(x) - \alpha_i\right)^2 \ \geq \ \delta(\epsilon) \quad \text{provided} \quad \|x\| \leq \rho(\epsilon) $$
and hence
\begin{equation*}
\begin{split} &\sigma^m \int_{x \in {\mathbb R}^n:\ \|x\| \leq \rho(\epsilon)} \exp\left\{ -\frac{\sigma^2}{2} \sum_{i=1}^m \left(q_i(x)-\alpha_i\right)^2 \right\} \ dx \\ \leq \ 
&\sigma^m \rho^n(\epsilon) \nu_n \exp\left\{ -\frac{\sigma^2 \delta(\epsilon)}{2} \right\}, 
\end{split}
\end{equation*}
where $\nu_n$ is the volume of the unit ball in ${\mathbb R}^n$. Therefore, there is $\sigma_0(\epsilon) > 0$ such that for all $\sigma > \sigma_0(\epsilon)$, we have 
\begin{equation}\label{eq2.3.5}
\sigma^m \int_{x \in {\mathbb R}^n:\ \|x\| \leq \rho(\epsilon)} \exp\left\{ -\frac{\sigma^2}{2} \sum_{i=1}^m \left(q_i(x)-\alpha_i\right)^2 \right\} \ dx \ \leq \ \frac{\epsilon}{2}.
\end{equation}
Combining (\ref{eq2.3.4}) and (\ref{eq2.3.5}), we conclude that the limit of the left hand side of (\ref{eq2.3.3}) is $0$ as $\sigma \longrightarrow +\infty$, which is the desired contradiction.
\end{proof}

The proof of Theorem \ref{th2.4} is similar.

\begin{proof}[Proof of Theorem \ref{th2.4}]
 Seeking a contradiction, suppose that the only solution to the system is $x=0$. Then for some $\delta >0$ we have 
 \begin{equation}\label{eq2.4.1}
\sum_{i=1}^m q_i^2(x) \ \geq \ \delta \quad \text{for all} \quad x \in {\mathbb R}^n \quad \text{such that} \quad \|x\|=1. 
\end{equation}
From Lemma \ref{le2.2}, for any $\sigma >0$, we have 
\begin{equation}\label{eq2.4.2}
\begin{split}
 & \sigma^m  \int_{{\mathbb R}^n} \exp\left\{ - \frac{\sigma^2}{2} \sum_{i=1}^m q_i^2(x)\right\} e^{-\|x\|^2/2} \ dx \\= &(2 \pi)^{\frac{n-m}{2}} \int_{{\mathbb R}^m} 
\det^{\quad -\frac{1}{2}} \left(I - \sqrt{-1} \sum_{i=1}^m \tau_i Q_i \right) e^{-\frac{\|t\|^2}{2 \sigma^2}} \ dt.  
\end{split}
\end{equation}
From (\ref{eq2.4.1}), the left hand side of (\ref{eq2.4.2}) is bounded above (we use polar coordinates) by
$$\omega_n \sigma^m \int_0^{+\infty} \exp\left\{ -\frac{\delta \sigma^2 \tau^2}{2} \right\} \tau^{n-1} e^{-\tau^2/2} \ d \tau,$$
where $\omega_n$ is the surface area of the unit sphere in ${\mathbb R}^n$. Using the substitution $\xi = \sigma \tau$, we rewrite the integral as 
$$\omega_n \sigma^{m-n}  \int_0^{+\infty} \exp\left\{-\frac{ \delta \xi^2}{2} \right\} \xi^{n-1} e^{-\xi^2/2 \sigma^2} \ d \xi$$
and observe that it converges to $0$ as $\sigma \longrightarrow +\infty$ (recall that $m < n$). On the other hand, the right hand side of (\ref{eq2.4.2}) converges to 
$$(2 \pi)^{\frac{n-m}{2}} \int_{{\mathbb R}^m} \det^{\quad -\frac{1}{2}} \left(I -\sqrt{-1} \sum_{i=1}^m \tau_i Q_i \right) \ne 0,$$
which is the desired contradiction.
\end{proof}

In the rest of the paper, we deduce Theorem \ref{th1.3} from Theorem \ref{th2.1} and Theorem \ref{th1.5} from Theorem \ref{th2.4}.

\section{Controlling eigenvalues}\label{sec4}

\subsection{Preliminaries}\label{subsec4.1} In the space of $n \times n$ real matrices we consider the standard inner product, see Section \ref{subsec1.3}. The corresponding Euclidean norm is called the {\it Hilbert-Schmidt} or {\it Frobenius} norm:
$$\|A\|_{\mathrm{HS}}=\sqrt{\langle A, A \rangle} = \sqrt{\tr (A^{\ast} A)}.$$
If, in addition, $A$ is symmetric with eigenvalues $\lambda_1, \ldots, \lambda_n$, we have 
$$\|A\|_{\mathrm{HS}}=\sqrt{\sum_{j=1}^n \lambda_j^2}$$
while for the operator norm we have 
$$\|A\|_{\mathrm{op}}= \max_{j=1, \ldots, n} |\lambda_j|.$$
We will also consider the 4-Schatten norm defined by
$$\|A\|_{\mathrm{S_4}}=\left(\sum_{j=1}^n \lambda_j^4 \right)^{1/4}.$$
This is indeed a norm in the space of $n \times n$ symmetric matrices, see, for example, Chapter 1 of \cite{Ta12}. In particular, we will use that 
\begin{equation}\label{eq4.1.1}
  |\ \|A\|_{\mathrm{S_4}}-\|B\|_{\mathrm{S_4}}\ | \ \leq \ \|A-B\|_{\mathrm{S_4}}.
\end{equation} 
Also, we observe that for a symmetric matrix $A$ with eigenvalues $\lambda_1, \ldots, \lambda_n$, we have 
$$\sum_{j=1}^n \lambda_j^4 \ \leq \ \left( \max_{j=1, \ldots, n} \lambda_j^2 \right) \sum_{j=1}^n \lambda_j^2,$$
from which it follows that 
\begin{equation}\label{eq4.1.2}
\|A\|_{\mathrm{S_4}} \ \leq \ \|A\|_{\mathrm{op}}^{1/2} \|A\|_{\mathrm{HS}}^{1/2}.
\end{equation}

Suppose that $B$ is a positive semidefinite symmetric matrix with eigenvalues $\lambda_1. \ldots, \lambda_n$. Then 
$$\|B\|^2_{\mathrm{HS}}= \sum_{j=1}^n \lambda_j^2 \ \leq \ \left( \max_{j=1, \ldots, n} \lambda_j \right) \sum_{j=1}^n \lambda_j= \|B\|_{\mathrm{op}}\left( \tr B\right).$$
We will apply the inequality in the following situation: Let $A_1, \ldots, A_m$ be an orthonormal set of symmetric matrices, so that 
$$\langle A_i, A_j \rangle =\tr (A_i A_j) =\begin{cases} 1 &\text{if\ } i=j, \\ 0 &\text{if\ } i \ne j. \end{cases}$$
Then the matrix 
$$B=\sum_{i=1}^m A_i^2$$
is symmetric positive semidefinite and hence we have 
\begin{equation}\label{eq4.1.3}
\left\| \sum_{i=1}^m A_i^2 \right\|_{\mathrm{HS}}^2 \ \leq \ m \left\| \sum_{i=1}^m A_i^2 \right\|_{\mathrm{op}}.
\end{equation}

We also remark that $\langle A, B \rangle \geq 0$ for any two $n \times n$ symmetric positive semidefinite matrices.

We will use the following inequality. Let $A_1, \ldots, A_m$ be an orthonormal set of $n \times n$ symmetric matrices and let $B$ be another $n \times n$, not necessarily symmetric, real matrix. Then 
$$\langle A_i, B \rangle = \tr (A_i B) \quad \text{for} \quad i=1, \ldots, m$$
are the coordinates of the orthogonal projection of $B$ onto $\spa\left(A_1, \ldots, A_m\right)$ and hence 
\begin{equation}\label{eq4.1.4}
\sum_{i=1}^m \tr^2(A_i B) \ \leq \ \|B\|_{\mathrm{HS}}^2.
\end{equation}

Finally, we will need moments of a random vector $w \in {\mathbb S}^{m-1}$, $w=\left(\omega_1, \ldots, \omega_m\right)$. Namely, for integer $\alpha_1, \ldots, \alpha_m \geq 0$, we have 
\begin{equation}\label{eq4.1.5}
\EE \omega_1^{\alpha_1} \cdots \omega_m^{\alpha_m} =0 \quad \text{provided at least one $\alpha_i$ is odd}
\end{equation}
and 
$$\EE \omega_1^{\alpha_1} \cdots \omega_m^{\alpha_m} = \frac{\Gamma\left(\frac{m}{2}\right) \prod_{i=1}^m \Gamma\left( \beta_i + \frac{1}{2}\right)}{\Gamma^m\left(\frac{1}{2}\right) \Gamma\left( \beta_1 + \ldots + \beta_m + \frac{m}{2}\right)} \quad \text{provided} \quad \alpha_i=2 \beta_i \quad \text{are even,}$$
see, for example, \cite{Ba02b}. 
In particular, we will use the following values:
\begin{equation}\label{eq4.1.6}
\begin{split} 
\EE \omega_i^2 \omega_j^2 = &\frac{1}{m(m+2)} \quad \text{for} \quad 1 \leq i \ne j \leq m, \\
\EE \omega_i^4 =&\frac{3}{m(m+2)} \quad \text{for} \quad i=1, \ldots, m, \\
\EE \omega_i^2 \omega_j^2 \omega_k^2 = &\frac{1}{m(m+2)(m+4)} \quad \text{for distinct} \quad 1 \leq i, j, k \leq m, \\
\EE \omega_i^2 \omega_j^4 = &\frac{3}{m(m+2)(m+4)} \quad \text{for} \quad 1 \leq i \ne j \leq m \quad \text{and} \\
\EE \omega_i^6 = &\frac{15}{m(m+2)(m+4)} \quad \text{for} \quad i=1, \ldots, m.
\end{split}
\end{equation}

In what follows, we fix an orthonormal set $A_1, \ldots, A_m$ of $n \times n$ symmetric matrices. For a random $w \in {\mathbb S}^{m-1}$, 
$w=\left(\omega_1, \ldots, \omega_m\right)$, sampled from the Haar probability measure in ${\mathbb S}^{m-1}$, we define
$$A(w)=\sum_{i=1}^m \omega_i A_i$$
and let $\lambda_1(w), \ldots, \lambda_n(w)$ be the eigenvalues of $A(w)$. Here is our first estimate.
\begin{lemma}\label{le4.2}
We have 
$$\EE \left( \sum_{j=1}^n \lambda_j^3(w)\right)^2 \ \leq \ \frac{120}{(m+2)(m+4)} \left\| \sum_{i=1}^m A_i^2 \right\|_{\mathrm{op}}.$$
\end{lemma}
\begin{proof} We have 
\begin{equation*}
\begin{split}
&\sum_{j=1}^n \lambda_j^3(w)=\tr\left( \sum_{i=1}^m \omega_i A_i \right)^3=\sum_{(i, j, k) \text{\ distinct}} \omega_i \omega_j \omega_k \tr (A_i A_j A_k) \\
&\qquad + \sum_{(i, j):\ i \ne j}  \omega_i^2 \omega_j \tr (A_i^2 A_j) + \sum_{(i, j):\ i \ne j} \omega_i \omega_j^2 \tr (A_i A_j^2) \\
&\qquad +\sum_{(i, j): \ i \ne j} \omega_i^2 \omega_j \tr (A_i A_j A_i)+ \sum_{i=1}^m \omega_i^3 \tr A_i^3 \\
=&\sum_{(i, j, k) \text{\ distinct}} \omega_i \omega_j \omega_k \tr (A_i A_j A_k) +3 \sum_{(i, j):\ i \ne j} \omega_i^2 \omega_j \tr (A_i^2 A_j) + \sum_{i=1}^m \omega_i^3 \tr A_i^3.
\end{split}
\end{equation*}
Using (\ref{eq4.1.5}) and (\ref{eq4.1.6}), we write 
$$\EE \left( \sum_{j=1}^n \lambda^3(w)\right)^2 = \frac{T_1 +27 T_2 + 15 T_3 + 18 T_4 + 9T_5}{m(m+2)(m+4)},$$
where
\begin{equation*}
\begin{split}
T_1 =& \sum_{\substack{(i, j, k) \text{\ distinct} \\ (i_1, j_1, k_1) \text{\ is a permutation of\ } (i_, j, k)}} \tr (A_i A_j A_k) \tr (A_{i_1} A_{j_1} A_{k_1}) \\
T_2=&\sum_{(i, j):\ i \ne j} \tr^2 (A_i^2 A_j) \\
T_3=&\sum_{i=1}^m \tr^2 (A_i^3) \\
T_4=&\sum_{(i, j):\ i \ne j} \tr (A_i^2 A_j) \tr (A_j^3) \quad \text{and}\\
T_5=&\sum_{(i, j, k) \text{\ distinct}} \tr (A_i^2 A_j) \tr (A_k^2 A_j).
\end{split}
\end{equation*}
Next, we bound $T_1$, $T_2$, $T_3$, $T_4$ and $T_5$.

Applying (\ref{eq4.1.4}) with $B=A_j A_k$, we obtain 
$$\sum_{i=1}^m \tr^2(A_i A_j A_k) \ \leq \ \|A_j A_k\|_{\mathrm{HS}}^2 = \tr (A_k A_j^2 A_k)=\tr (A_j^2 A_k^2)$$
and hence
$$\sum_{(i, j, k) \text{\ distinct}} \tr^2(A_i A_j A_k) \ \leq \ \sum_{(j, k):\ j \ne k} \tr (A_j^2 A_k^2) \ \leq \ \left\| \sum_{i=1}^m A_i^2 \right\|_{\mathrm{HS}}^2.$$
By the Cauchy - Schwarz inequality, for every permutation $\sigma$ of $\{1, 2, 3\}$, we obtain 
\begin{equation*}
\begin{split} &\left| \sum_{(i_1, i_2, i_3) \text{\ distinct}} \tr \left(A_{i_1} A_{i_2} A_{i_3}\right) \tr \left( A_{i_{\sigma(1)}} A_{i_{\sigma(2)}} A_{i_{\sigma(3)}}\right) \right| \ \leq \ 
\sum_{(i, j, k) \text{\ distinct}} \tr^2 (A_i A_j A_k) \\ &\qquad \leq \left\| \sum_{i=1}^m A_i^2 \right\|^2_{\mathrm{HS}}
\end{split}
\end{equation*}
and hence 
$$|T_1| \ \leq \ 6 \left\| \sum_{i=1}^m A_i^2 \right\|_{\mathrm{HS}}^2.$$
Applying (\ref{eq4.1.4}) with $B=A_i^2$, we conclude that 
\begin{equation*}
\sum_{j=1}^m \tr^2(A_i^2 A_j) = \sum_{j=1}^m \tr^2 (A_j A_i^2) \ \leq \ \| A_i^2\|^2_{\mathrm{HS}}
\end{equation*}
and hence 
$$|T_2| \ \leq \ \sum_{i=1}^m \left\| A_i^2\right\|^2_{\mathrm{HS}} \ \leq \ \left\| \sum_{i=1}^m A_i^2 \right\|^2_{\mathrm{HS}},$$
where the last inequality follows since the matrices $A_1^2, \ldots, A_m^2$ are symmetric positive semidefinite and hence 
$$\langle A_i^2, A_j^2\rangle \geq 0 \quad \text{for all} \quad i, j.$$
Applying the Cauchy - Schwarz inequality, we obtain 
\begin{equation}\label{eq4.2.1}
| \tr A_i^3| =| \langle A_i, A_i^2 \rangle| \ \leq \ \|A_i\|_{\mathrm{HS}} \|A_i^2\|_{\mathrm{HS}}=\|A_i^2\|_{\mathrm{HS}}
\end{equation}
and hence 
$$|T_3| \ \leq \ \sum_{i=1}^m \|A_i\|^2_{\mathrm{HS}} \ \leq \ \left\| \sum_{i=1}^m A_i^2 \right\|^2_{\mathrm{HS}}.$$
To bound $T_4$ and $T_5$ we combine some of the previously obtained estimates. 

Applying the Cauchy - Schwarz inequality, (\ref{eq4.1.4}) with $B=\sum_{i=1}^m A_i^2$ and (\ref{eq4.2.1}), we obtain
\begin{equation*}
\begin{split}
&\left| \sum_{j=1}^m \sum_{i=1}^m \tr (A_i^2 A_j) \tr (A_j^3) \right| = \left| \sum_{j=1}^m \tr \left(A_j \sum_{i=1}^m A_i^2 \right) \tr (A_j^3) \right| \\
&\qquad \leq \left| \sum_{j=1}^m \tr^2 \left(A_j \sum_{i=1}^m A_i^2 \right) \right|^{1/2} \left| \sum_{j=1}^m \tr^2 (A_j^3)\right|^{1/2} \\
&\qquad \leq \left\|\sum_{i=1}^m A_i^2 \right\|_{\mathrm{HS}} \left( \sum_{j=1}^m \left\| A_j^2 \right\|^2_{\mathrm{HS}}\right)^{1/2} \ \leq \ \left\| \sum_{i=1}^m A_i^2 \right\|^2_{\mathrm{HS}}.
\end{split}
\end{equation*}
Therefore, using (\ref{eq4.2.1}), we get 
$$|T_4| = \left| \sum_{j=1}^m \sum_{i=1}^m \tr (A_i^2 A_j) \tr (A_j^3) - \sum_{i=1}^m \tr^2 (A_i^3) \right| \ \leq \ 2 \left\| \sum_{i=1}^m A_i^2 \right\|^2_{\mathrm{HS}}.$$
It remains to bound $T_5$. We have 
\begin{equation*}
\begin{split}
T_5 =&\sum_{j=1}^m  \sum_{\substack{(i, k): \\ i \ne j, k \ne j}} \tr (A_i^2 A_j) \tr (A_k^2 A_j) - \sum_{j=1}^m \sum_{i:\ i \ne j} \tr^2(A_i^2 A_j) \\
=&\sum_{j=1}^m \left(\sum_{i:\ i \ne j} \tr(A_i^2 A_j) \right)^2 - T_2.
\end{split}
\end{equation*}
Since 
$$0 \ \leq \ T_2 \ \leq \ \left\| \sum_{i=1}^m A_i^2 \right\|^2_{\mathrm{HS}},$$
we have 
$$|T_5| \ \leq \ \max\left\{ \sum_{j=1}^m \left(\sum_{i: i \ne j} \tr(A_i^2 A_j) \right)^2, \quad  \left\| \sum_{i=1}^m A_i^2\right\|^2_{\mathrm{HS}} \right\}.$$
Now,
\begin{equation*}
\begin{split}
&\left(\sum_{i:\ i \ne j} \tr (A_i^2 A_j)\right)^2=\left( -\tr(A_j^3) + \sum_{i=1}^m \tr(A_i^2 A_j) \right)^2 \\
&\qquad = \tr^2 (A_j^3) -2 \tr (A_j^3) \sum_{i=1}^m \tr (A_i^2 A_j) + \left(\sum_{i=1}^m \tr(A_i^2 A_j)\right)^2 \\
&\qquad =\tr^2 (A_j^3) -2 \tr (A_j^3)  \tr\left(A_j \sum_{i=1}^m A_i^2\right)+ \tr^2\left( A_j \sum_{i=1}^m A_i^2 \right)
\end{split}
\end{equation*}
and hence
\begin{equation*}
\begin{split}
&\sum_{j=1}^m \left(\sum_{i: i \ne j} \tr(A_i^2 A_j) \right)^2\\&\qquad = \sum_{j=1}^m \tr^2 (A_j^3)  -2 \sum_{j=1}^m \tr (A_j^3) \tr\left(A_j \sum_{i=1}^m A_i^2 \right) +
\sum_{j=1}^m  \tr^2 \left(A_j \sum_{i=1}^m A_i^2\right)
\end{split}
\end{equation*}
By (\ref{eq4.2.1}), we get
$$\sum_{j=1}^m \tr^2 (A_j^3) \ \leq \ \sum_{j=1}^m \|A_j^2\|^2_{\mathrm{HS}} \ \leq \ \left\| \sum_{i=1}^m A_i^2 \right\|^2_{\mathrm{HS}}.$$
Then, from the  the Cauchy - Schwarz inequality, (\ref{eq4.2.1}) and (\ref{eq4.1.4}) with $B=\sum_{i=1}^m A_i^2$, we get
\begin{equation*}
\begin{split}
&\left| \sum_{j=1}^m \tr (A_j^3) \tr \left(A_j \sum_{i=1}^m A_i^2 \right) \right| \ \leq \ \left(\sum_{j=1}^m \tr^2 (A_j^3) \right)^{1/2} 
\left(\sum_{j=1}^m \tr^2 \left(A_j \sum_{i=1}^m A_i^2 \right) \right)^{1/2} \\ &\quad \leq \left\|\sum_{i=1}^m A_i^2 \right\|_{\mathrm{HS}}^2
\end{split}
\end{equation*}
and from (\ref{eq4.1.4})
$$\sum_{j=1}^m  \tr^2 \left(A_j \sum_{i=1}^m A_i^2\right) \ \leq \ \left\| \sum_{i=1}^m A_i^2 \right\|^2_{\mathrm{HS}}.$$
Thus 
$$|T_5| \ \leq \ 4 \left\| \sum_{i=1}^m A_i^2 \right\|^2_{\mathrm{HS}}.$$
Summarizing,
$$\EE \left(\sum_{j=1}^n \lambda_j^3(w)\right)^2 \ \leq \ \frac{120}{m(m+2)(m+4)} \left\| \sum_{i=1}^m A_i^2 \right\|^2_{\mathrm{HS}} \ \leq \
\frac{120}{(m+2)(m+4)} \left\| \sum_{i=1}^m A_i^2 \right\|_{\mathrm{op}},$$
where the last inequality follows by (\ref{eq4.1.3}).
\end{proof}

Next, we bound the 4th moment of the eigenvalues.
\begin{lemma}\label{le4.3}
We have
$$\EE \left(\sum_{j=1}^n \lambda_j^4(w)\right) \ \leq \ \frac{3}{m+2} \left\|\sum_{i=1}^m A_i^2 \right\|_{\mathrm{op}}.$$
\end{lemma}
\begin{proof} Using (\ref{eq4.1.5}) and (\ref{eq4.1.6}), we write
\begin{equation*}
\begin{split}
&\EE\left(\sum_{j=1}^n \lambda_j^4(w)\right) = \EE \tr \left(\sum_{i=1}^m \omega_i A_i \right)^4 \\&\quad = \EE \left( \sum_{(i, j):\ i \ne j} \omega_i^2 \omega_j^2 \tr (A_i A_j A_i A_j) \right) + \EE \left( \sum_{(i, j):\ i \ne j} \omega_i^2 \omega_j^2 \tr (A_i^2 A_j^2)  \right) \\&\qquad \qquad + \EE \left( \sum_{(i, j):\ i \ne j} \omega_i^2 \omega_j^2 \tr (A_i A_j^2 A_i) \right) + 
\EE \left(\sum_{i=1}^m \omega_i^4 \tr (A_i^4)\right) \\
&\quad =\frac{T_1 + 2 T_2 + 3 T_3}{m(m+2)},
\end{split}
\end{equation*}
where
\begin{equation*}
\begin{split}
T_1=&\sum_{(i, j):\ i \ne j} \tr (A_i A_j A_i A_j), \\
T_2=&\sum_{(i, j):\ i \ne j} \tr (A_i^2 A_j^2) \quad \text{and} \\
T_3=&\sum_{i=1}^m \tr (A_i^4).
\end{split} 
\end{equation*}
We bound $T_1$, $T_2$ and $T_3$. 

Applying the Cauchy - Schwarz inequality, we get 
\begin{equation*}
\begin{split}
&|T_1|=\left| \sum_{(i, j):\ i \ne j} \tr (A_i A_j A_i A_j) \right| = \left| \sum_{(i, j):\ i \ne j} \langle A_j A_i, A_i A_j \rangle \right| \ \leq \ \sum_{(i, j):\ i \ne j} \|A_j A_i\|_{\mathrm{HS}} \| A_i A_j \|_{\mathrm{HS}} \\ &\qquad= \sum_{(i, j):\ i \ne j} \tr (A_i^2 A_j^2) =T_2.
\end{split}
\end{equation*} 
On the other hand,
\begin{equation*}
\begin{split}
T_2= \sum_{(i, j):\ i \ne j} \tr (A_i^2 A_j^2) =\tr \left( \sum_{i=1}^m A_i^2 \right)^2 -\sum_{i=1}^m \tr (A_i^4) = \left\| \sum_{i=1}^m A_i^2 \right\|^2_{\mathrm{HS}} - T_3.
\end{split}
\end{equation*}
Therefore,
$$| T_1 +2 T_2 + 3T_3| \leq |T_1| + 2 T_2 + 3T_3 \ \leq \ 3 T_2 + 3 T_3 = 3 \left\| \sum_{i=1}^m A_i^2 \right\|^2_{\mathrm{HS}}$$
The proof now follows by (\ref{eq4.1.3}).
\end{proof}

Next, we prove some uniform bounds.
\begin{lemma}\label{le4.4} For all $w \in {\mathbb S}^{m-1}$, we have
\begin{enumerate}
\item 
$$\| A(w)\|_{\mathrm{op}} \ \leq \ \left\| \sum_{i=1}^m A_i^2 \right\|^{1/2}_{\mathrm{op}} \quad \text{and} $$
\item $$\sum_{j=1}^n \lambda_j^4(w) \ \leq \ \left\| \sum_{i=1}^m A_i^2 \right\|_{\mathrm{op}}.$$
\end{enumerate}
\end{lemma} 
\begin{proof} Repeatedly applying the Cauchy - Schwarz inequality, for any vector $x \in {\mathbb R}^n$ such that $\|x\|=1$, we obtain
\begin{equation*}
\begin{split}
| \langle A(w) x, x \rangle| = &\left| \sum_{i=1}^m \omega_i \langle A_i x, x \rangle \right|  \ \leq \ \left(\sum_{i=1}^m \langle A_i x, x \rangle^2 \right)^{1/2} \ \leq \
\left(\sum_{i=1}^m \langle A_i x, A_i x \rangle \right)^{1/2} \\
=& \left\langle \left(\sum_{i=1}^m A_i^2\right) x, x \right\rangle^{1/2}  \ \leq \ \left\| \sum_{i=1}^m A_i^2 \right\|^{1/2}_{\mathrm{op}},
\end{split}
\end{equation*}
and Part (1) follows. Note that here we did not use that $A_1, \ldots, A_m$ is an orthonormal set.

To prove Part (2), we bound
$$\sum_{j=1}^n \lambda_j^4(w) \ \leq \ \left( \max_{j=1, \ldots, n} \lambda_j^2(w)\right) \sum_{j=1}^n \lambda_j^2(w) =\left\| A(w)\right\|^2_{\mathrm{op}} \| A(w)\|^2_{\mathrm{HS}}.$$
Using that $A_1, \ldots, A_m$ is an orthonormal set, we obtain
\begin{equation}\label{eq4.4.1}
\|A(w)\|^2_{\mathrm{HS}} = \tr (A^2(w)) = \sum_{i, j=1}^m \omega_i \omega_j \tr (A_i A_j)= \sum_{i=1}^m \omega_i^2 =1.
\end{equation}
The proof now follows by Part (1).
\end{proof} 

Finally, we need a concentration inequality on the unit sphere ${\mathbb S}^{m-1}$ for the $4$-Schatten norm of $A(w)$.
\begin{lemma}\label{le4.5} For $\delta  \geq 0$, we have 
\begin{equation*}
\begin{split}
&\PP\left\{ w \in {\mathbb S}^{m-1}: \ \left\| A(w)\right\|_{\mathrm{S_4}} \ \geq \ \left( \frac{3}{m+2} \left\| \sum_{i=1}^m A_i^2 \right\|_{\mathrm{op}}\right)^{1/4} +\delta \right\} \\
&\quad  \leq \ \exp\left\{ -\frac{\delta^2 (m-1)}{2 \left\| \sum_{i=1}^m A_i^2 \right\|_{\mathrm{op}}^{1/2}}\right\}.
\end{split}
\end{equation*}
\end{lemma} 
\begin{proof} 
We apply a measure concentration inequality on the sphere ${\mathbb S}^{m-1}$. Let 
$$\dist(x, y)=\arccos \langle x, y \rangle$$
be the geodesic distance between two points $x, y \in {\mathbb S}^{m-1}$ and let $f: {\mathbb S}^{m-1} \longrightarrow {\mathbb R}$ be a 1-Lipschitz function, so that 
$$|f(x)-f(y)| \ \leq \ \dist(x, y) \quad \text{for all} \quad x, y \in {\mathbb S}^{m-1}.$$
Then for $c=\EE f$ and $\delta >0$ we have 
$$\PP\left\{w \in {\mathbb S}^{m-1}: \ f(w) \ \geq \ c+ \delta \right\} \ \leq \ \exp\left\{ -\frac{\delta^2 (m-1)}{2} \right\},$$
see, for example, Section 5.1 of \cite{Le01}.

Let us define a function $g: {\mathbb R}^m \longrightarrow {\mathbb R}$ by
$$g(x)=\|A(x)\|_{\mathrm{S_4}}, \quad \text{where} \quad A(x)=\sum_{i=1}^m \xi_i A_i \quad \text{for} \quad x=\left(\xi_1, \ldots, \xi_m\right).$$
Then from (\ref{eq4.1.1}) and Part 2 of Lemma \ref{le4.4}, for all $x, y \in {\mathbb S}^{m-1}$, we have 
\begin{equation*}
\begin{split}
|g(x) -g(y)| \ \leq \ &\|A(x)-A(y)\|_{\mathrm{S_4}}=\|A(x-y)\|_{\mathrm{S_4}} \ \leq \ \left\| \sum_{i=1}^m A_i^2 \right\|^{1/4}_{\mathrm{op}} \|x-y\| \\ \leq \ 
&\left\| \sum_{i=1}^m A_i^2 \right\|^{1/4}_{\mathrm{op}} \dist(x, y).
\end{split}
\end{equation*}
Therefore, for the expectation $c=\EE g$ on the unit sphere ${\mathbb S}^{m-1}$, we have
$$\PP\left\{ w \in {\mathbb S}^{m-1}: \ g(w) \ \geq c + \delta \right\} \ \leq \ \exp\left\{ - \frac{\delta^2 (m-1)}{2 \left\| \sum_{i=1}^m A_i^2 \right\|^{1/2}_{\mathrm{op}}} \right\} \quad \text{for} \quad \delta \geq 0.$$
By Lemma \ref{le4.3} and the H\"older inequality, we get 
$$c = \EE \left( \sum_{j=1}^n \lambda_j^4(w)\right)^{1/4}  \ \leq \ \left( \EE \sum_{j=1}^n \lambda_j^4(w) \right)^{1/4} \ \leq \ 
\left( \frac{3}{m+2} \left\| \sum_{i=1}^m A_i^2 \right\|_{\mathrm{op}} \right)^{1/4},$$
and the proof follows.
\end{proof} 

\section{Estimating integrals}\label{sec5}

Recall that we have an orthonormal set $A_1, \ldots, A_m$ of $n \times n$ symmetric real matrices. For $w \in {\mathbb S}^{m-1}$, 
$w=\left(\omega_1, \ldots, \omega_m\right)$, we define the matrix 
$$A(w)=\sum_{i=1}^m \omega_i A_i.$$ As follows by (\ref{eq4.4.1}), we have 
$$\| A(w)\|_{\mathrm{HS}} =1.$$

In this section, we consider the integral 
$$\int_0^{+\infty} \tau^{m-1} \det^{\quad -\frac{1}{2}} \left(I - \sqrt{-1} \tau A \right) \exp\left\{-\frac{\sqrt{-1} \tau}{2} \tr A\right\} \ d \tau,$$
where $A$ is an $n \times n$ symmetric matrix satisifying $\|A\|_{\mathrm{HS}}=1$ and possibly some other constraints. In particular, we will be interested in the situation when 
$$\|A\|_{\mathrm{op}}=O\left( \frac{1}{\sqrt{m}}\right).$$

We will be comparing this integral with 
$$\int_0^{+\infty} \tau^{m-1} \exp\left\{-{\frac{\tau^2}{4}} \right\} \ d \tau =2^{m-1} \Gamma \left(\frac{m}{2}\right) \ \sim \left(\frac{2}{e}\right)^{m/2} m^{m/2}.$$
First, we bound the tail.

\begin{lemma}\label{le5.1} Let $A$ be an $n \times n$ real symmetric matrix such that 
$$\|A\|_{\mathrm{HS}} =1 \quad \text{and} \quad \|A\|_{\mathrm{op}} \ \leq \ \frac{ 1}{10\sqrt{m}}.$$
Then for $m \geq 2$,
$$\int_{5 \sqrt{m}}^{+\infty} \tau^{m-1} \left| \det^{\quad-\frac{1}{2}} \left(I- \sqrt{-1} \tau A\right)\right| \ d \tau \ <  \ \frac{1}{20m} m^{m/2} e^{-3m}.$$
\end{lemma}

\begin{proof} Let $\lambda_1, \ldots, \lambda_n$ be the eigenvalues of $A$, so that 
$$\sum_{j=1}^n \lambda_j^2 =1 \quad \text{and} \quad |\lambda_j| \ \leq \ \frac{\alpha}{\sqrt{m}} \quad \text{for} \quad j=1, \ldots, n$$
(we will choose $\alpha=0.1$ at the end).
Then 
$$\left| \det^{\quad -\frac{1}{2}} \left( I-\sqrt{-1} \tau I\right)\right|=  \prod_{j=1}^n \left| 1- \sqrt{-1} \tau \lambda_j \right|^{-\frac{1}{2}} = \prod_{j=1}^n 
\left(1+ \lambda_j ^2 \tau^2\right)^{-\frac{1}{4}}.$$
Let 
$$\xi_j=\lambda_j^2 \tau^2 \quad \text{for} \quad j=1, \ldots, n.$$
Since the minimum of the log-concave function $\prod_{j=1}^n (1+\xi_j)$ on the convex polyhedron defined by the equation 
$$\sum_{j=1}^n \xi_j =\tau^2$$ 
and inequalities 
$$0 \ \leq \ \xi_j \ \leq \ \frac{\alpha^2 \tau^2}{m} \quad \text{and} \quad j=1, \ldots, n$$
is attained at its vertex where all but possibly one coordinate are either $0$ or $\alpha^2 \tau^2/m$, we have 
$$\prod_{j=1}^n \left(1 + \lambda_j^2 \tau^2\right)^{-\frac{1}{4}} \ \leq \ \left( 1+ \frac{\alpha^2 \tau^2}{m} \right)^\frac{\alpha^2 - m}{4 \alpha^2}.$$
Hence
\begin{equation*}
\begin{split}
&\int_\frac{\sqrt{m}}{2 \alpha}^{+\infty} \tau^{m-1} \left| \det^{\quad -\frac{1}{2}} \left(I - \sqrt{-1} \tau A \right)\right| \ d \tau \ \leq \ \int_\frac{\sqrt{m}}{2 \alpha}^{+\infty} \tau^{m-1} \left( 1+ \frac{\alpha^2 \tau^2}{m} \right)^\frac{\alpha^2 - m}{4 \alpha^2} \ d \tau \\ 
&\quad = \frac{m^{m/2}}{\alpha^m} \int_{{1/2}}^{+\infty} s^{m-1} (1+s^2)^\frac{\alpha^2 - m}{4 \alpha^2} \ d s \ \leq \ \frac{m^{m/2}}{\alpha^m} 
\int_{1/2}^{+\infty} (1+s^2)^\frac{m-1}{2} (1+s^2)^\frac{\alpha^2 - m}{4 \alpha^2} (2s) \ ds \\
&\quad =\frac{m^{m/2}}{\alpha^m} \int_{1/2}^{+\infty} (1+s^2)^{\frac{-m(1-2 \alpha^2)-\alpha^2}{4 \alpha^2}} (2s) \ ds \\ 
&\quad= \frac{m^{m/2}}{\alpha^m} \frac{4 \alpha^2}{(1-2 \alpha^2)m - 3 \alpha^2} \left(\frac{4}{5}\right)^\frac{(1-2\alpha^2)m -3 \alpha^2}{4 \alpha^2}
\end{split}
\end{equation*}
Substituting $\alpha=0.1$, we get
\begin{equation*}
\begin{split}
&\int_{5 \sqrt{m}}^{+\infty} \tau^{m-1} \left| \det^{\quad-\frac{1}{2}} \left(I - \sqrt{-1} \tau A \right) \right| \ d \tau \ \leq \ m^{m/2} 10^m \frac{0.04}{0.98m-0.03}
\left(\frac{4}{5}\right)^{(24.5) m -0.75} \\ &\quad < \ \frac{1}{20 m} 10^m m^{m/2} \left(\frac{4}{5}\right)^{(24.5)m} \ < \ \frac{1}{20 m} m^{m/2} e^{-3m}.
\end{split}
\end{equation*}
\end{proof}

Next, we estimate the integral on the initial interval.
\begin{lemma}\label{le5.2} Let $A$ be an $n \times n$ real symmetric matrix such that 
$$\|A\|_{\mathrm{HS}}=1 \quad \text{and} \quad \|A\|_{\mathrm{op}} \ \leq \ \frac{1}{10 \sqrt{m}}$$
and let $\lambda_1, \ldots, \lambda_n$ be the eigenvalues of $A$. Then, for $m \geq 1$,
\begin{enumerate}
\item We have 
\begin{equation*}
\begin{split}
&\int_0^{5 \sqrt{m}} \tau^{m-1} \left| \det^{\quad -\frac{1}{2}} \left(I - \sqrt{-1} \tau A \right) \right| \ d \tau \\ &\quad \leq \ \exp\left\{ \frac{625 m^2}{8} \sum_{j=1}^n \lambda_j^4 \right\} 
\int_0^{5 \sqrt{m}} \tau^{m-1} \exp\left\{ - {\frac{\tau^2}{4}}\right\}\ d \tau.
\end{split}
\end{equation*}
\item Suppose, in addition, that 
$$\left| \sum_{j=1}^n \lambda_j^3 \right| \ \leq \ \frac{1}{25 m^{3/2}} \quad \text{and} \quad \sum_{j=1}^n \lambda_j^4 \ \leq \ \frac{1}{625 m^2}.$$
Then 
\begin{equation*}
\begin{split}
&\Re \int_0^{5 \sqrt{m}} t^{m-1} \det^{\quad -\frac{1}{2}} (I - \sqrt{-1} \tau A) \exp\left\{ -\frac{\sqrt{-1} \tau}{2} \tr A \right\} \ d \tau \\ 
&\quad \geq \  \frac{1}{2}  \int_0^{5 \sqrt{m}} 
\tau^{m-1} \exp\left\{ -{\frac{\tau^2}{4}}\right\} \ d \tau. 
\end{split}
\end{equation*}
\end{enumerate}
\end{lemma}

\begin{proof} Since in the interval $0 \leq \tau \leq 5 \sqrt{m}$, we have 
$$|\tau \lambda_j| \ \leq \ \frac{1}{2} \quad \text{for} \quad j=1, \ldots, n,$$
we can expand 
\begin{align} \label{eq: det}
&\det^{\quad -\frac{1}{2}} \left(I - \sqrt{-1} \tau A\right) \exp\left\{ -\frac{ \sqrt{-1} \tau}{2} \tr A \right\} \\
&\quad = \exp\left\{ -\frac{1}{2} \sum_{j=1}^n \ln (1- \sqrt{-1} \tau \lambda_j) -
\frac{\sqrt{-1} \tau}{2} \sum_{j=1}^n \lambda_j \right\} \notag \\
&\quad = \exp\left\{ \frac{1}{2} \sum_{k=2}^{\infty} \frac{(\tau \sqrt{-1})^k}{k} \sum_{j=1}^n \lambda_j^k \right\} \notag \\ 
&\quad = \exp\left\{h(\tau) + \sqrt{-1} g(\tau) \right\}, \notag
\end{align} 
where 
$$h(\tau)=\sum_{s=1}^{\infty} (-1)^s \frac{\tau^{2s}}{4s} \sum_{j=1}^n \lambda_j^{2s} \quad \text{and} \quad g(\tau) =\sum_{s=1}^{\infty} (-1)^s \frac{\tau^{2s+1}}{4s+2} \sum_{j=1}^n \lambda_j^{2s+1}.$$
We have 
$$\sum_{j=1}^n \lambda_j^2 =\|A\|^2_{\mathrm{HS}} = 1$$
and for $s \geq 	1$, we have 
$$\sum_{j=1}^n \lambda_j^{2(s+1)} \ \leq \ \left(\max_{j=1, \ldots, n} \lambda_j^2 \right) \sum_{j=1}^n \lambda_j^{2s} \ \leq \ \frac{1}{100 m} \sum_{j=1}^n \lambda_j^{2s}.$$
Consequently, for $0 \leq \tau \leq 5 \sqrt{m}$, we have
$$\sum_{j=1}^n (\tau \lambda_j)^{2(s+1)} \ \leq \ \frac{\tau^2}{100 m} \sum_{j=1}^n (\tau \lambda_j)^{2s} \ \leq \ \frac{1}{4} \sum_{j=1}^n (\tau \lambda_j)^{2s}.$$
Hence the terms of $h(\tau)$ alternate in sign and decrease in the absolute value, from which we deduce that 
\begin{equation}\label{eq5.2.1}
-{\frac{\tau^2}{4}} \ \leq \ h(\tau) \ \leq \ -{\frac{\tau^2}{4}} + \frac{\tau^4}{8}  \sum_{j=1}^n \lambda_j^4  \quad \text{for} \quad 0 \ \leq \ \tau \ \leq \ 5 \sqrt{m}.
\end{equation}
Part (1)  now follows from the upper bound in (\ref{eq5.2.1}).

To prove Part (2), we bound $g(\tau)$ assuming that 
$$\left| \sum_{j=1}^n \lambda_j^3 \right| \ \leq \ \frac{\alpha}{m^{3/2}} \quad \text{and} \quad \sum_{j=1}^n \lambda_j^4 \ \leq \ \frac{\beta}{m^2}$$
(we substitute $\alpha=1/25$ and $\beta=1/625$  at the end). 
For $s \geq 2$, we have
$$\sum_{j=1}^n |\lambda_j|^{2s+1}\ \leq \ \left(\max_{j=1, \ldots, n} |\lambda_j| \right)^{2s-3} \cdot  \sum_{j=1}^n \lambda_j^4 \  \leq \ \frac{\beta}{10^{2s-3} m^{s+\frac{1}{2}}} .$$
Therefore, in the interval $0 \leq \tau \leq 5 \sqrt{m}$, we have
$$\left| \sum_{j=1}^n (\tau \lambda_j)^3 \right| \ \leq \ 125 \alpha \quad \text{and} \quad \left| \sum_{j=1}^n (\tau \lambda_j)^{2s+1} \right| \ \leq \ \frac{625 \beta}{2^{2s-3}} \quad 
\text{for} \quad s \geq 2.$$
Therefore, in the interval $0 \leq \tau \leq 5 \sqrt{m}$, we have 
$$|g(\tau)| \ \leq \ \frac{125 \alpha}{6} + \sum_{s=2}^{\infty} \frac{625 \beta}{(4s+2) 2^{2s-3}}  \ \leq  \ \frac{125 \alpha}{6} + \frac{1250 \beta}{30}.$$
Substituting 
$$\alpha= \frac{1}{25} \quad \text{and} \quad \beta=\frac{1}{625},$$
we conclude that 
$$|g(\tau)| \ \leq \ \frac{5}{6} +\frac{1}{15} =\frac{27}{30} \ < \ \frac{\pi}{3} \quad \text{for all} \quad 0 \leq \tau \leq 5 \sqrt{m}.$$
The proof now follows from the lower bound in (\ref{eq5.2.1}).
\end{proof} 

The last lemma of this section contains some estimates for our benchmark integral.
\begin{lemma}\label{le5.3}  For $m \geq 2$, we have 
\begin{enumerate}
\item
$$\int_0^{+\infty} \tau^{m-1} \exp\left\{ -{\frac{\tau^2}{4}}\right\} \ d \tau \ \geq \ m^{m/2} \sqrt{\frac{\pi}{m}} \left( \frac{2}{e}\right)^{m/2} \quad \text{and}$$
\item
\begin{equation*}
\begin{split}
&\int_{5 \sqrt{m}}^{+\infty} \tau^{m-1} \exp\left\{ -{\frac{\tau^2}{4}}\right\} \ d \tau \\
&\quad \leq \ \sqrt{\frac{2 \pi}{m-1}} 2^m m^{m/2} \exp\left\{ - \frac{25(m-1)}{8}\right\}. 
\end{split}
\end{equation*}
\end{enumerate} 
\end{lemma}

\begin{proof} We have
\begin{equation*}
\begin{split}
\int_0^{+\infty} \tau^{m-1} \exp\left\{ -{\frac{\tau^2}{4}}\right\} \ d \tau=2^{m-1} \int_0^{+\infty} s^{\frac{m-2 }{2}} \exp\{ -s \} \ ds = 2^{m-1} \Gamma\left(\frac{m}{2}\right). 
\end{split}
\end{equation*}
To prove Part (1), we use the standard inequality 
$$\Gamma(x) \ \geq \ \sqrt{2 \pi} x^{x-\frac{1}{2}} e^{-x} \quad \text{for} \quad x \geq 1.$$

To prove Part (2), we bound
\begin{equation*}
\begin{split}
&\int_{5 \sqrt{m}}^{+\infty} \tau^{m-1} \exp\left\{ -{\frac{\tau^2}{4}}\right\} \ d \tau =2^m m^{m/2} \int_{5/2}^{+\infty} s^{m-1} \exp\left\{ -ms^2 \right\} \ ds  \\
&\quad \leq \ 2^m m^{m/2} \int_{5/2}^{+\infty} \exp\left\{ -(m-1) \left(s^2 - \ln s \right) \right\} \ ds  \\ &\quad \leq \ 2^m m^{m/2} \int_{5/2}^{+\infty} \exp\left\{ -\frac{(m-1) s^2 }{2} \right\} 
\ ds = \frac{2^m m^{m/2}}{\sqrt{m-1}} \int_{\frac{5\sqrt{m-1}}{2}}^{+\infty} \exp\left\{ - \frac{\tau^2}{2} \right\} \ d \tau \\ &\quad \leq \ 
\sqrt{\frac{2 \pi}{ m-1}} 2^m m^{m/2} \exp\left\{ -\frac{25(m-1)}{8}\right\},
\end{split}
\end{equation*}
where in the last inequality we use the standard Gaussian probability tail estimate
$$\frac{1}{\sqrt{2 \pi}} \int_a^{+\infty} e^{-\tau^2/2} \ d \tau \ \leq \ e^{-a^2/2} \quad \text{for} \quad a \geq 0.$$
\end{proof}

\section{Proofs of Theorems \ref{th1.3} and \ref{th1.5}}\label{sec6}

\begin{proof}[Proof of Theorem \ref{th1.3}] 
We choose 
$$\eta=10^{-6}.$$

Let $A_1, \ldots, A_m$ be an orthonormal basis of the subspace $\spa\left(Q_1, \ldots, Q_m\right)$ in the space of $n \times n$ 
symmetric matrices and let
$$a_i(x)=\langle A_i x, x \rangle \quad \text{for} \quad i=1, \ldots, m$$
be the corresponding quadratic forms. Since the quadratic forms $q_1, \ldots, q_m$ are linear combinations of the forms $a_1, \ldots, a_m$ and vice versa, the system 
$$q_i(x)=\tr Q_i \quad \text{for} \quad i=1, \ldots, m$$ 
has a solution if and only if the system 
$$a_i(x)=\tr A_i \quad \text{for} \quad i=1, \ldots, m$$
has a solution $x$. To establish the existence of a solution of the latter system, we use Theorem \ref{th2.1}, for which we consider the integral 
\begin{equation}\label{eq6.1}
\int_{{\mathbb R}^m} \det^{\quad -\frac{1}{2}} \left(I - \sqrt{-1} \sum_{i=1}^m \tau_i A_i \right) \exp\left\{ -\frac{\sqrt{-1}}{2} \sum_{i=1}^m \tau_i \tr A_i \right\} \ dt.
\end{equation}

Our goal is to prove that the integral (\ref{eq6.1}) converges absolutely to a non-zero value. 

Let ${\mathbb S}^{m-1} \subset {\mathbb R}^m$ be the unit sphere endowed with the Haar probability measure. For $w \in {\mathbb S}^{m-1}$, $w=\left(\omega_1, \ldots, \omega_m\right)$, 
let 
$$A(w)=\sum_{i=1}^m \omega_i A_i.$$
Then by (\ref{eq4.4.1}) and Lemma \ref{le4.4} for every $w \in {\mathbb S}^{m-1}$, we have 
$$\| A(w)\|_{\mathrm {HS}}=1 \quad \text{and} \quad \|A\|_{\mathrm{op}}=\sqrt{ \frac{\eta}{m}} \ < \ \frac{1}{10\sqrt{m}}.$$
It follows from Lemma (\ref{le5.1}) that 
$$\int_0^{+\infty} \tau^{m-1} \left| \det^{\quad -\frac{1}{2}} \left(I - \sqrt{-1} \tau A(w)\right)\right| \ d \tau \ < \ + \infty.$$
Hence the integral (\ref{eq6.1}) indeed converges absolutely and, up to a non-zero factor (the surface area of the sphere ${\mathbb S^{m-1}}$) can be written as an absolutely converging integral
\begin{equation}\label{eq6.2}
\EE \left( \int_{0}^{+\infty}  \det^{\quad -\frac{1}{2}} \left(I - \sqrt{-1} \sum_{i=1}^m \tau A(w) \right) \exp\left\{ -\frac{\sqrt{-1} \tau}{2} \tr A(w)\right\} \ d\tau\right)
\end{equation}
where the expectation is taken with respect to the Haar measure on ${\mathbb S}^{m-1}$.
Hence our goal is to prove that the integral (\ref{eq6.2}) is non-zero. We intend to prove that the real part of the integral is positive. 

Let $\lambda_1(w), \ldots, \lambda_n(w)$ be the eigenvalues of $A(w)$. By Lemma \ref{le4.2},
$$\EE \left(\sum_{j=1}^n \lambda_j^3(w) \right)^2 \ \leq \ \frac{120 \eta}{m (m+2) (m+4)}  \ < \ \frac{3}{25000 m^3}.$$
Therefore, by the Markov inequality, 
\begin{equation}\label{eq6.3}
\PP\left\{w: \ \left| \sum_{j=1}^n \lambda_j^3(w) \right| \ > \ \frac{1}{25 m^{3/2}}\right\} \ \leq \  \frac{3}{40}.
\end{equation}
By Lemma \ref{le4.3}, 
$$\EE \left(\sum_{j=1}^n \lambda_j^4(w) \right) \ \leq \ {\frac{3\eta}{m(m+2)}}   \ < \ \frac{3}{10^6 m^2}.$$
Hence, using  the Markov inequality again, we get
\begin{equation}\label{eq6.4}
\PP\left\{w: \sum_{j=1}^n \lambda_j^4(w) \ > \ \frac{1}{625 m^2} \right\} \ \leq \ \frac{3}{1600}.
\end{equation}
We represent ${\mathbb S}^{m-1}$ as a disjoint union 
$${\mathbb S}^{m-1} = \Omega_0 \cup \Omega_1 \cup \Omega_2,$$
where
$$\Omega_0=\left\{ w \in {\mathbb S}^{m-1}:\quad  \left| \sum_{j=1}^n \lambda_j^3(w) \right| \ \leq \ \frac{1}{25 m^{3/2}} \quad \text{and} \quad \sum_{j=1}^n \lambda_j^4(w) \ \leq \ \frac{1}{625 m^2} \right\},$$
$$\Omega_1 =\left\{w \in {\mathbb S}^{m-1}:\quad \left| \sum_{j=1}^n \lambda_j^3(w) \right| \ > \ \frac{1}{25 m^{3/2}} \quad \text{and} \quad \sum_{j=1}^n \lambda_j^4(w) \ \leq \ 
\frac{1}{625 m ^2} \right\} \quad \text{and} $$
$$\Omega_2=\left\{ w \in {\mathbb S}^{m-1}:\quad  \sum_{j=1}^n \lambda_j^4(w)  \ > \ \frac{1}{625 m^2} \right\}.$$
From (\ref{eq6.3}) and (\ref{eq6.4}), we have 
$$\PP( \Omega_0) \ \geq \ \frac{7}{8}$$
and hence from Part (2) of Lemma \ref{le5.2},
\begin{equation}\label{eq6.5}
\begin{split}
&\Re \int_{\Omega_0} \left( \int_{0}^{5 \sqrt{m}}  \det^{\quad -\frac{1}{2}} \left(I - \sqrt{-1} \sum_{i=1}^m \tau A(w) \right) \exp\left\{ -\frac{\sqrt{-1}\tau}{2} \tr A(w)\right\} \ d\tau\right) \ dw \\
&\quad \geq \ \frac{7}{16} \int_0^{5 \sqrt{m}} \tau^{m-1} \exp\left\{ - {\frac{\tau^2}{4}} \right\} \ d \tau,
\end{split}
\end{equation}
where $dw$ is the Haar measure in ${\mathbb S}^{m-1}$. 

From (\ref{eq6.3}), we have 
$$P(\Omega_1) \ \leq \ \frac{3}{40} $$
and hence Part (1) of Lemma \ref{le5.2} yields
\begin{equation}\label{eq6.6}
\begin{split}
&\int_{\Omega_1} \left(\int_0^{5 \sqrt{m}} \tau^{m-1} \left| \det^{\quad -\frac{1}{2}} \left(I - \sqrt{-1} \tau A(w)\right)\right| \ d \tau \right) \ d w \\
&\quad \leq \frac{3}{40} \exp\left\{\frac{1}{8}\right\} \int_0^{5 \sqrt{m}} \tau^{m-1} \exp \left\{- {\frac{\tau^2}{4}}\right\} \ d \tau \\
&\quad <  0.1 \int_0^{5 \sqrt{m}} \tau^{m-1} \exp \left\{- {\frac{\tau^2}{4}}\right\} \ d \tau. 
\end{split} 
\end{equation}

For integer $k \geq 1$, let 
$$\Omega_2^k=\left\{w \in {\mathbb S}^{m-1}: \quad \frac{k}{5 \sqrt{m}} \ < \ \|A(w)\|_{\mathrm{S_4}} \ \leq \ \frac{k+1}{5 \sqrt{m}}\right\}.$$
Then from Part (2) of Lemma \ref{le4.4}, we have 
$$\Omega_2= \bigcup_{k=1}^{5(\eta m)^{1/4}}  \Omega_2^k.$$
By Lemma \ref{le4.5}, taking into account that $\eta=10^{-6}$, we get 
\begin{equation*}
\begin{split}
&\PP\left(\Omega_2^k\right) \ \leq \ \PP\left\{w\in {\mathbb S}^{m-1}: \ \|A(w)\|_{\mathrm{S_4}}\  \geq \  \left(\frac{3\eta}{m(m+2)}\right)^{1/4}+ \frac{k}{6 \sqrt{m}} \right\} \\
&\quad \leq \  \exp\left\{ -\frac{k^2(m-1)}{72 \sqrt{\eta m} } \right\}.
\end{split}
\end{equation*}
In view of Part (1) of Lemma \ref{le5.2},
\begin{equation*}
\begin{split}
&\int_{\Omega_2^k} \left( \int_0^{5 \sqrt{m}} \tau^{m-1} \left| \det^{\quad -\frac{1}{2}} \left(I - \sqrt{-1} \tau A\right) \right| \ d \tau \right)\ d w \\
&\quad \leq  \exp\left\{ \frac{(k+1)^4}{8}  -\frac{k^2(m-1)}{72 \sqrt{\eta m} } \right\} \int_0^{5 \sqrt{m}} \tau^{m-1} \exp\left\{ - {\frac{\tau^2}{4}}\right\} \ d \tau.
\end{split}
\end{equation*} 
Since 
$$k \ \leq \ 5(\eta m)^{1/4},$$
we have 
\begin{equation*}
\frac{(k+1)^4}{8}  \ \leq \ 2 k^4 \ \leq 50 k^2 (\eta m)^{1/2}
\end{equation*} 
and 
\begin{equation*}
\frac{(k+1)^4}{8}  -\frac{k^2(m-1)}{72 \sqrt{\eta m} } \ \leq \ \frac{k^2 \sqrt{m}}{20} - 6 k^2 \sqrt{m}  \ < \ -5 k^2.
\end{equation*} 
Hence
\begin{equation}\label{eq6.7}
\begin{split}
&\int_{\Omega_2} \left( \int_0^{5 \sqrt{m}} \tau^{m-1} \left| \det^{\quad -\frac{1}{2}} \left(I - \sqrt{-1} \tau A\right) \right| \ d \tau \right)\ d w \\
&\quad < \ \left(\sum_{k=1}^{\infty} \exp\left\{ -5 k^2 \right\} \right) \int_0^{5 \sqrt{m}} \tau^{m-1} \exp\left\{ - {\frac{\tau^2}{4}}\right\} \ d \tau \\
&\quad < \ 0.01 \int_0^{5 \sqrt{m}} \tau^{m-1} \exp\left\{ - {\frac{\tau^2}{4}}\right\} \ d \tau.
\end{split}
\end{equation}
Summarizing, from (\ref{eq6.5}), (\ref{eq6.6}) and (\ref{eq6.7}), we get 
\begin{equation*}
\begin{split}
&\left| \EE  \int_0^{5 \sqrt{m}} \det^{\quad -\frac{1}{2}} \left(I - \sqrt{-1} \sum_{i=1}^m \tau A(w) \right) \exp\left\{ -\frac{\sqrt{-1}\tau}{2} \tr A(w)\right\} \ d\tau \right| \\ &\quad  > \ 
\frac{1}{4} \int_0^{5 \sqrt{m}} \tau^{m-1} \exp\left\{-{\frac{\tau^2}{4}}\right\} \ d \tau
\end{split}
\end{equation*}
and hence by Lemma \ref{le5.1}, the absolute value of the expectation (\ref{eq6.2}) is at least 
\begin{equation*}
\frac{1}{4} \int_0^{5 \sqrt{m}} \tau^{m-1} \exp\left\{ -{\frac{\tau^2}{4}}\right\} - \frac{1}{20 m} m^{m/2} e^{-3m}.
\end{equation*} 
Then by Lemma \ref{le5.3}, the absolute value of the expectation (\ref{eq6.2}) is at least 
\begin{equation*}
m^{m/2} \left( \sqrt{\frac{\pi }{16m}} \left(\frac{2}{e}\right)^{m/2} - \sqrt{\frac{\pi}{m-1}} 2^m \exp\left\{ -\frac{25(m-1)}{8} \right\} -\frac{1}{20 m} e^{-3m}   \right),
\end{equation*} 
which is positive for $m \geq 3$.
\end{proof}

\begin{proof}[Proof of Theorem \ref{th1.5}] The proof is identical, except we use Theorem \ref{th2.4} instead of Theorem \ref{th2.1}.
\end{proof}


\begin{thebibliography}{GGM}

\bibitem{Ba93}
{
A. Barvinok, Feasibility testing for systems of real quadratic equations, Discrete $\&$ Computational Geometry {\bf 10} (1993), no.~1, 1--13.
}

\bibitem{Ba02a}
{
A. Barvinok, {\it A Course in Convexity}, Graduate Studies in Mathematics, {\bf 54}, American Mathematical Society, Providence, RI, 2002.
}

\bibitem{Ba02b}
{
A. Barvinok, Estimating $L^{\infty}$ norms by $L^{2k}$ norms for functions on orbits, Foundations of Computational Mathematics {\bf 2} (2002), no.~4, 393--412.
}

\bibitem{Ba08}
{
S. Basu, Computing the top Betti numbers of semialgebraic sets defined by quadratic inequalities in polynomial time, Foundations of Computational Mathematics {\bf 8} (2008), no.~1, 45--80.
}

\bibitem{B+06}
{
S. Basu, R. Pollack and M.-F. Roy, {\it Algorithms in Real Algebraic Geometry}, second edition, Algorithms and Computation in Mathematics, {\bf 10}, Springer-Verlag, Berlin, 
2006.
}

\bibitem{Bi16}
{
D. Bienstock, A note on polynomial solvability of the CDT problem, SIAM Journal on Optimization {\bf 26} (2016), no.~1, 488--498.
}

\bibitem{Bu07}
{
P. B\"urgisser, Average Euler characteristic of random real algebraic varieties,  Comptes Rendus Math\'ematique. Acad\'emie des Sciences. Paris {\bf 345} (2007), no.~9, 507--512.
}

\bibitem{GP05}
{
D. Grigoriev and D.V. Pasechnik, Polynomial-time computing over quadratic maps. I. Sampling in real algebraic sets, Computational Complexity {\bf 14} (2005), no.~1, 20--52.
}

\bibitem{GV88}
{
D. Yu. Grigor'ev and N.N. Vorobjov, Jr., Solving systems of polynomial inequalities in subexponential time, Journal of Symbolic Computation, {\bf 5} (1988), no.~1-2, 37--64.
}

\bibitem{Le01}
{ 
M. Ledoux, {\it The Concentration of Measure Phenomenon}, Mathematical Surveys and Monographs, {\bf 89}, American Mathematical Society, Providence, RI, 2001.
}

\bibitem{LL16}
{
A. Lerario and E. Lundberg, Gap probabilities and Betti numbers of a random intersection of quadrics, Discrete $\&$ Computational Geometry, {\bf 55}(2016), no.~2, 462--496.
}

\bibitem{L+14}
{
L. Liberti, C. Lavor, N. Maculan and A. Mucherino, Euclidean distance geometry and applications, SIAM Review {\bf 56} (2014), no.~1, 3--69.
}

\bibitem{Pa13}
{
P.A. Parrilo, Semidefinite optimization, in: {\it Semidefinite Optimization and Convex Algebraic Geometry}, MOS-SIAM Series on Optimization, {\bf 13}, SIAM,  Philadelphia, PA, 2013, pp. 3--46.
}

\bibitem{Ta12}
{
T. Tao, {\it Topics in Random Matrix Theory}, Graduate Studies in Mathematics, {\bf 132}, American Mathematical Society, Providence, RI, 2012.
}

\bibitem{We78}
{
A. Wehrl, General properties of entropy, Reviews of Modern Physics {\bf 50} (1978), no.~2, 221--260.
}

\end{thebibliography}
\end{document}